\documentclass{mcom-l}

\newtheorem{theorem}{Theorem}[section]
\newtheorem{lemma}[theorem]{Lemma}
\newtheorem{corollary}[theorem]{Corollary}

\theoremstyle{definition}

\theoremstyle{remark}
\newtheorem{remark}[theorem]{Remark}
\numberwithin{equation}{section}

\usepackage{stix}
\usepackage{ulem}
\usepackage{bm}
\usepackage{graphicx,color}
\usepackage{float}
\usepackage{subfig}
\usepackage{hyperref}
\hypersetup{
    colorlinks=true,      
    urlcolor=black,
}

\def\D{\mathcal{D}}
\def\F{\mathcal{F}}
\def\P{\mathbb{P}}
\def\E{\mathbb{E}}

\def\({\left(}
\def\){\right)}

\def\<{(}
\def\>{)}

\begin{document}

\title[]{Strong convergence of a fully discrete finite element method for a class of semilinear stochastic partial differential equations with multiplicative noise}

\author{Xiaobing Feng}
\address{Department of Mathematics, The University of Tennessee, Knoxville, TN 37996, U.S.A. }
\thanks{The work of the first author was partially supported by the NSF grant: DMS-1318486 }
\email{xfeng@math.utk.edu}
\author{Yukun Li}
\address{Department of Mathematics, The Ohio State University, Columbus, OH 43210, U.S.A.}
\email{li.7907@osu.edu}
\author{Yi Zhang}
\address{Department of Mathematics and Statistics, The University of North Carolina at Greensboro, Greensboro, NC 27402, U.S.A.}
\email{y\_zhang7@uncg.edu}
 


\date{}

\dedicatory{}

\begin{abstract}
This paper develops and analyzes a fully discrete finite element method for a class of semilinear stochastic 
partial differential equations (SPDEs) with 
multiplicative noise. The nonlinearity in the diffusion term of the SPDEs is assumed to be globally 
Lipschitz and the nonlinearity in the drift term is only assumed to satisfy a one-side Lipschitz condition. 
These assumptions are the same ones as used in \cite{higham2002strong} where numerical methods 
for general nonlinear stochastic ordinary 
differential equations (SODEs) under ``minimum assumptions'' were studied. As a result, the semilinear SPDEs 
considered in this paper is a direct generalization of the SODEs considered in \cite{higham2002strong}.  
There are several difficulties which need to be overcome for this generalization.  First, obviously the spatial  discretization, which does not appear in the SODE case, adds an extra layer of difficulty. It turns out a special 
discretization must be designed to guarantee certain properties for the numerical scheme and its stiffness 
matrix. In this paper we use a finite element interpolation technique to discretize the nonlinear drift 
term. Second, in order to prove the strong convergence of the proposed fully discrete finite element 
method,  stability estimates for higher order moments of the $H^1$-seminorm of the numerical solution must be
established, which are difficult and delicate. A judicious combination of the properties of the drift 
and diffusion terms and a nontrivial technique borrowed from \cite{majee2018optimal} is used in this paper to 
achieve the goal. Finally, stability estimates for the second and higher order moments of
the $L^2$-norm of the numerical solution is also difficult to obtain due to the fact that the 
mass matrix may not be diagonally dominant. This is done by utilizing the interpolation theory 
and the higher moment estimates for the $H^1$-seminorm of the numerical solution. 
After overcoming these difficulties, it is proved that the proposed fully discrete finite element method
is convergent in strong norms with nearly optimal rates of convergence.  Numerical experiment results 
are also presented to validate the theoretical results and to demonstrate the efficiency of the 
proposed numerical method. 
\end{abstract}

\maketitle


\section{Introduction}\label{sec-1}
We consider the following initial-boundary value problem for general semilinear stochastic partial 
differential equations (SPDEs) with function-type multiplicative noise:
\begin{alignat}{2}\label{sac_s}
d u &= \bigl[ \Delta u + f(u) \bigr] \, dt + g(u)\, d W(t), &&\qquad\mbox{in } \D \times (0,T),\\
\frac{\partial u}{\partial \nu} &=0,  &&\qquad\mbox{on } \partial \D\times (0,T), \label{sac_s2} \\
u(\cdot, 0) &= u_0(\cdot),  &&\qquad\mbox{in } \D. \label{sac_s3}
\end{alignat}
Here $\D\subset \mathbf{R}^d (d=1,2,3)$ is a bounded domain, $W: \Omega\times (0,T)\to \mathbf{R}$ 
denotes the standard Weiner process on the filtered probability space $(\Omega,\F,\{\F_t : t\ge0\}, \P)$, and $f,g\in C^1$ are two given functions and $f(u)$ takes the form
\begin{align}\label{eq20180812_1} 
f(u)=c_0u-c_1u^3-c_2u^5-c_3u^7-\cdots,
\end{align}
where $c_i\ge 0, i=0,1,2,\cdots$. For the sake of clarity, we only consider the case $f(u)=u-u^q$ 
in this paper, where $q\ge3$ is an odd integer (it is trivial when $f(u) = c_0u$). We remark that similar results still hold for the 
general nonlinear function $f(u)$ in \eqref{eq20180812_1}, and when $f(u)=\frac{1}{\epsilon^2}(u-u^3)$, 
\eqref{sac_s} is known as the stochastic Allen-Cahn equation with function-type multiplicative noise and 
interaction length $\epsilon$ \cite{majee2018optimal}. We also assume that $g$ is 
globally Lipschitz, that is, there exists a constant $\kappa_1>0$ such that
\begin{align}\label{eq20180812_2}
|g(a)-g(b)|\le \kappa_1|a-b|.
\end{align}
Setting $b=0$ in \eqref{eq20180812_2}, we get
\begin{align}\label{eq20180813_1}
|g(a)|^2&\le C+Ca^2,\\
|g(a)\,a|&\le C+Ca^2.\label{eq20180813_2}
\end{align}

Under the above assumptions for the drift term and the diffusion term, it can be proved that 
 \cite{gess2012strong} there exists a unique strong variational solution u
such that
\begin{align}\label{var_solu}
(u(t),\phi) &= (u(0),\phi) - \int_0^t \Bigl(\nabla u(s), \nabla \phi \Bigr) \, ds +\int_0^t  \big( f(u(s)),\phi \big) \, ds\\
   &\quad   
+ \int_0^t (g(u),\phi) \, d W(s) \qquad \forall \, \phi \in H^1(\D) \notag 
\end{align}
holds $\P$-almost surely. Moreover, when the initial condition $u_0$ is sufficiently smooth, the following stability estimate for the strong solution $u$ holds: 
\begin{align}\label{eq20190919_1}
\sup_{t\in[0,T]}\E \left[ \|\Delta u(t)\|_{L^2}^2 \right] + \sup_{t\in[0,T]} \E \left[ \|u(t)\|^{2q}_{L^{2q}} \right]
+\sup_{t\in[0,T]}\E \left[ \|u(t)\|_{L^{\infty}}^{2q-2} \right]\le C.
\end{align}

Clearly, when the $\Delta u$ term in \eqref{sac_s} is dropped, the PDE reduces to a stochastic ODE.  
A convergence theory for numerical approximations for this stochastic ODE was established long ago 
(cf. \cite{mao2007stochastic, kloeden1991numerical}) under the global Lipschitz assumptions on $f$ and $g$. 
Later, the convergence was proved in \cite{higham2002strong} under a weaker condition on $f$ known as 
a one-side Lipschitz condition in the sense that there exists a constant $\mu>0$ such that
\begin{align}\label{eq20180918_1}
(a-b,f(a)-f(b))\le\mu(a-b)^2\qquad\forall a,b\in\mathbb{R}.
\end{align}
The optimal rate of convergence was also obtained in \cite{higham2002strong}
under an extra assumption that $f$ behaves like a polynomial. The one-side Lipschitz condition is widely used and it has broad applications \cite{burrage1979stability, butcher1975stability, dahlquist1976error, dekker1984stability, stuart1998dynamical}.

We also note that numerical approximations of the SPDE \eqref{sac_s} with various special drift terms 
and/or diffusion terms have been extensively investigated in the literature, see \cite{feng2014finite, feng2017finite, majee2018optimal, prohl2014strong}. In particular, we mention that the case that 
$f(u)=u-u^3$, $g(u), g'(u), 
g''(u)$ are bounded and $g(u)$ is global Lipschitz continuous was studied in \cite{majee2018optimal}, 
the high moments of the $H^1$-norm of the numerical solution were proved to be stable, and 
a nearly optimal strong convergence rate was established. A specially designed discretization is used for $f(u)=u-u^3$, and it is not trivial to extend the idea to the case when $f(u)=u-u^q$ where $q>3$.  

The goal of this paper is to generalize the numerical SODE theory of \cite{higham2002strong}
to the SPDE case. Specifically,
we want to design a fully discrete finite element method for problem \eqref{sac_s}--\eqref{sac_s3}
which can be proved to be stable and convergent with optimal rates in strong norms 
under ``minimum'' assumptions on nonlinear functions $f$ and $g$ as those used 
in \cite{higham2002strong}. We recall that
the ``minimum'' assumptions refer to that $g$ is assumed to be global Lipschitz,  and 
$f$ satisfies the one-side Lipschitz condition \eqref{eq20180918_1} and it behaves like a polynomial.
To the best of our knowledge, such a goal has yet been achieved  before in the literature. 

The remainder of this paper is organized as follows. In Section \ref{sec-2}, we establish several 
 H\"older continuity properties (in different norms) for the SPDE solution $u$ and for the composite function
$f(u)$. These properties play an important role in our error analysis. In Section \ref{sec-3}, we first 
present our fully discrete finite element method for problem \eqref{sac_s}--\eqref{sac_s3}, which
consists of an Euler-type scheme for time discretization and a nonstandard finite element method 
for spatial discretization. The novelty of our spatial discretization is to approximate the nonlinear function $f$ 
by its  finite element interpolation in the scheme. We then establish several key properties for the 
numerical solution,  among them are the stability of the second and higher order moments of its $H^1$-seminorm 
and the stability of the second and higher order moments of its $L^2$-norm. We note that the proofs of 
the stability of these higher order moments are quite involved, and they require some special techniques and 
rely on the structure of the proposed numerical method. For example, the diagonal dominance property of 
the stiffness matrix is needed to show the stability of the second and higher order moments of the $H^1$-seminorm
of the numerical solution, however, the mass matrix may not be diagonally dominant. To circumvent this difficulty,
we use the stability of the second and higher order moments of the $H^1$-seminorm of the numerical solution 
and the interpolation theory to get the desired $L^2$-norm stability. 
Finally, in this section we prove nearly optimal order error estimates for the numerical solution 
by utilizing the stability of
higher order moments of the $L^2$-norm and $H^1$-seminorm of the numerical solution. We like to emphasize that
only sub-optimal order error estimates could be obtained should the stability of higher order moments 
of the $H^1$-seminorm of the numerical solution were not known, see \cite{prohl2014strong} where the special case
$f(u)=u-u^3$ was considered.
In Section \ref{sec-4}, we present several numerical experiments to validate our theoretical results,
especially to verify the stability of numerical solution using different initial conditions $u_0$ and 
different functions $f$ and $g$. As a special case, the stochastic Allen-Cahn equation with 
function-type multiplicative noise is also tested.

\section{Preliminaries and properties of the SPDE solution}\label{sec-2} 
Throughout this paper, we shall use $C$ to denote a generic constant, and we take the standard Sobolev 
notations in \cite{BS2008}. When it is the whole domain $\D$, $\|\cdot\|_{H^k}$ and $\|\cdot\|_{L^p}$ 
are used to simplify $\|\cdot\|_{H^k(\D)}$ and $\|\cdot\|_{L^p(\D)}$ respectively, and $(\cdot\ ,\ \cdot)$ 
is used to denote the standard inner product of $L^2(\D)$. $\E[\cdot]$ denotes the 
expectation operator on the filtered probability space $(\Omega,\F,\{\F_t : t\ge0\}, \P)$. 

In this section, we first derive the H\"older continuity in time for the strong solution $u$ with respect 
to the spatial $H^1$-seminorm and for the composite function $f(u)$ with respect to the spatial $L^2$-norm. 
Both results will play a key role in the error analysis (see Subsection \ref{subsec3}). 
The time derivatives of $\nabla u$ and the composite function $f(u)$ do not exist in the stochastic case, 
so these H\"older continuity results will substitute for the differentiability of $\nabla u$ and $f(u)$ 
with respect to time in the error analysis.

\begin{lemma}\label{lem:e3}
Let $u$ be the strong solution to problem \eqref{var_solu}.
Then for any $s,t \in [0,T]$ with $s < t$, we have
\begin{align*}
\E \big[ \|\nabla (u(t)-u(s)) \|_{L^2}^2 \big] + \frac{1}{2} \E\left[ \int_{s}^t \|\Delta(u(\zeta)-u(s))\|_{L^2}^2 \, d\zeta \right] \leq C_1 (t-s),
\end{align*}
where
\begin{align*}
C_1= C\bigg(\sup_{s \leq \zeta \leq t} 
\E \left[ \|\Delta u(\zeta)\|_{L^2}^2 \right] 
+ \sup_{s \leq \zeta \leq t} \E \left[ \|u(\zeta)\|^{2q}_{L^{2q}} \right]+ \sup_{s \leq \zeta \leq t} \E \left[ \|u(\zeta)\|_{L^2}^2 \right]\bigg).
\end{align*}
\end{lemma}

\begin{proof}
Applying It\^{o}'s formula to the functional $\Phi(u(\cdot)) :=\|\nabla u(\cdot) 
- \nabla u(s)\|_{L^2}^2$ with fixed $s \in [0,T)$ and using integration by parts, we get
\begin{align}
\label{e3:1}
&\|\nabla u(t)-\nabla u(s)\|_{L^2}^2 
= -2 \int_{s}^t (\Delta u(\zeta)-\Delta u(s),\Delta u(\zeta) ) \, d\zeta \\
&\quad 
- 2\int_s^t \Bigl( \Delta u(\zeta) - \Delta u(s), f(u(\zeta)) \Bigr) \, d\zeta \notag \\
&\quad
-2\int_{s}^t ( \Delta u(\zeta) - \Delta u(s), g(u(\zeta)) ) \, d W(\zeta)
+\int_{s}^t \|\nabla g(u(\zeta))\|^2 \, d \zeta. \notag
\end{align}

The expectation of the first term on the right-hand side of \eqref{e3:1} can be bounded by the Cauchy-Schwarz inequality as follows
\begin{align} \label{e3:2}
&-2 \E \left[ \int_{s}^t (\Delta u(\zeta)-\Delta u(s),\Delta u(\zeta) ) \, d \zeta \right] \\
&\quad = -2 \E \left[ \int_s^t \|\Delta u(\zeta)-\Delta u(s) \|^2_{L^2} \, d\zeta 
+ \int_{s}^t (\Delta u(\zeta)-\Delta u(s),\Delta u(s) ) \, d \zeta \right] \notag \\
&\quad \leq - \E \left[ \int_s^t \|\Delta u(\zeta)-\Delta u(s)\|^2_{L^2} \, d\zeta \right] 
+ \E \left[  \|\Delta u(s)\|^2_{L^2}\right] (t-s). \notag
\end{align}
The expectation of the second term on the right-hand side of \eqref{e3:1} can be bounded by
\begin{align} \label{e3:3}
2\E &\left[ \int_s^t \Bigl( \Delta u(\zeta)-\Delta u(s), f(u(\zeta)) \Bigr) \, d\zeta \right] \\
&\quad 
\leq \E \left[ \int_s^t \left( \| \Delta u(\zeta)-\Delta u(s) \|_{L^2}^2 
+ \| f(u(\zeta)) \|_{L^2}^2 \right) \, d \zeta \right] \notag \\
&\quad 
\leq C \Bigl( \sup_{s \leq \zeta \leq t} 
\E \left[ \|\Delta u(\zeta)\|_{L^2}^2 \right] 
+ \sup_{s \leq \zeta \leq t} \E \left[ \|u(\zeta)\|^{2q}_{L^{2q}} \right]\notag\\
&\qquad+ \sup_{s \leq \zeta \leq t} \E \left[ \|u(\zeta)\|_{L^2}^2 \right] \Bigr) (t-s) \notag.
\end{align}

Next we bound the expectation of the fourth term on the right-hand side of \eqref{e3:1} as follows
\begin{align}\label{e3:5}
\E \left[\int_s^t \|\nabla g(u(\zeta))\|^2 d\zeta \right] 
\leq  C\sup_{s \leq \zeta \leq t} \E \left[ \| \nabla u(\zeta)\|^2_{L^2} \right] (t-s).
\end{align}
Then Lemma \ref{lem:e3} follows from \eqref{e3:1}--\eqref{e3:5} and the fact that the expectation of the third term on the 
right-hand side of \eqref{e3:1} is zero.
\end{proof}

Next we prove the H\"{o}lder continuity result for the nonlinear term $f(u(t))-f(u(s))$ with respect to the spatial $L^2$-norm.

\begin{lemma}\label{lem:e4}
Let $u$ be the strong solution to problem \eqref{var_solu}. Then for any $s,t \in [0,T]$ with $s < t$, we have
\begin{align*}
\E \big[ \|f(u(t))-f(u(s))\|_{L^2}^2 \big]\leq C_2 (t-s),
\end{align*}
where 
\begin{align*}
&C_2 = \Bigl( C+\sup_{s \leq \zeta \leq t} 
\E \left[\|\Delta u(\zeta)\|_{L^2}^2 \right]+ \sup_{s \leq \zeta \leq t} 
\E \left[\|u(\zeta)\|_{L^2}^2 \right]
+ \sup_{s \leq \zeta \leq t} \E \left[ \|u(\zeta)\|_{L^{2q}}^{2q} \right]\notag\\
&+ \sup_{s \leq \zeta \leq t} \E \left[ \|u(\zeta)\|_{L^{4}}^{4} \right]\Bigr)\times\Bigl(C+\sup_{s \leq \zeta \leq t} 
\E \left[ \|u(\zeta)\|_{L^{\infty}}^{q-2} \right]+\sup_{s \leq \zeta \leq t} 
\E \left[ \|u(\zeta)\|_{L^{\infty}}^{2q-2} \right]\Bigr).
\end{align*}
\end{lemma}

\begin{proof}
Applying It\^{o}'s formula to $\Phi(u(\cdot)) 
:= \|f(u(\cdot))-f(u(s))\|_{L^2}^2$ with fixed $s \in [0,T)$, we obtain
\begin{align}\label{e4:1}
&\|f(u(t))-f(u(s))\|_{L^2}^2 = 2\int_s^t \int_{\D} \big( f(u(\zeta))-f(u(s))\big)f'(u(\zeta)) \\
&\qquad 
\times \Bigl[ \Delta u(\zeta)+f(u(\zeta)) \Bigr] \, dx \, d\zeta\notag\\
&\quad 
+ 2 \int_s^t \int_{\D} \bigl( f(u(\zeta))-f(u(s)) \bigr) f'(u(\zeta))  g(u(\zeta)) \, dx \, d W(\zeta) \notag\\
&\quad 
+  \int_s^t \int_{\D} \bigl( f(u(\zeta))-f(u(s)) \bigr) f''(u(\zeta)) 
| g(u(\zeta)) |^2 \, dx \, d\zeta \notag\\
&\quad 
+  \int_s^t \int_{\D}[f'(u(\zeta))]^2| g(u(\zeta)) |^2 \, dx \, d\zeta. \notag
\end{align}

Taking the expectation on both sides, it follows from integration by parts and Young's inequality that
\begin{align}\label{e4:2}
& \E \left[ \|f(u(t))-f(u(s))\|_{L^2}^2 \right] 
\leq C(t-s)\times \Bigl( \sup_{s \leq \zeta \leq t} 
\E \left[\|\Delta u(\zeta)\|_{L^2}^2 \right]+\\
& \sup_{s \leq \zeta \leq t} 
\E \left[\|u(\zeta)\|_{L^2}^2 \right]
+ \sup_{s \leq \zeta \leq t} \E \left[ \|u(\zeta)\|_{L^{2q}}^{2q} \right]+ \sup_{s \leq \zeta \leq t} \E \left[ \|u(\zeta)\|_{L^{4}}^{4} \right]+C\Bigr)\notag\\
&\times\Bigl(\sup_{s \leq \zeta \leq t} 
\E \left[ \|u(\zeta)\|_{L^{\infty}}^{q-2} \right]+\sup_{s \leq \zeta \leq t} 
\E \left[ \|u(\zeta)\|_{L^{\infty}}^{2q-2} \right]+C\Bigr). \notag
\end{align}

Finally, the desired Lemma \ref{lem:e4} follows from \eqref{e4:2}.
\end{proof}

\begin{remark}
(a) For the diffusion term, the global Lipschitz condition, which is stronger than the one-side
  Lipschitz condition, is needed as in the SODE case. Using the $C^1$ assumption and the global
   Lipschitz assumption, we can derive that the derivative of the diffusion term is bounded by the 
   Lipschitz constant $\kappa$, i.e., $|g'(u)|\le\kappa$, but the diffusion term itself may not be 
   bounded. For instance, $g(u)=u$, $g(u)=\sqrt{u^2+1}$, etc. Notice these two assumptions are 
   consistent with the SODE case in \cite{higham2002strong}, and they are also the conditions 
   to guarantee the well-posedness \cite{higham2002strong} of the strong SODE solution;

(b) We can verify $f(u)$ in \eqref{eq20180812_1} satisfies a one-sided Lipschitz condition 
\eqref{eq20180918_1}. If the drift term $f$ behaves polynomially, then for the one-sided 
Lipschitz condition \eqref{eq20180918_1}, we have the following conclusions:

(1). The power $q$ of the highest order term must be odd. Because when the highest power $q$ is even, dividing $\pm c_q(a^q-b^q)$ by $a-b$ yields the the quotient is odd so that it can be $+\infty$ and $-\infty$. When choosing $a$ and $b$ sufficiently large or small, the absolute value of this term is dominant and the left-hand side of \eqref{eq20180918_1} is $C|a-b|^2$ where $C$ can be $+\infty$, which is a contradiction;

(2). The sign of the highest odd order term must be negative. Because this term is dominant and the quotient of dividing $c_q(a^q-b^q)$ by $a-b$ can be $+\infty$, which contradicts \eqref{eq20180918_1}.
\end{remark}

\section{Fully discrete finite element approximation}\label{sec-3}
\subsection{Formulation of the finite element method}\label{sec-3.1}
In this section, we first construct a fully discrete finite element method for problem 
\eqref{sac_s}--\eqref{sac_s3}. we then establish several stability properties for 
the numerical solution including the stability of higher order moments for 
its $H^1$-seminorm and $L^2$-norm. Finally, we derive optimal order error estimates in strong norms
for the numerical solution using the stability estimates.  

Let $t_n=n\tau\ (n = 0, 1, \ldots, N)$ be a uniform partition of $[0,T]$ and $\mathcal{T}_h$ 
be the triangulation of $\D$ satisfying the following assumption \cite{xu1999monotone}:
\begin{equation}\label{eq20180907}
\frac{1}{d(d-1)}\sum_{K\supset E}|\kappa_E^K|\cot\theta_E^K\ge0,
\end{equation}
where $E$ denotes the edge of simplex $K$. It was proved in \cite{xu1999monotone} that the stiffness matrix for the Poisson equation with zero Dirichlet boundary is an $M$-matrix if and only if this assumption holds for all edges. The stiffness matrix is diagonally dominant if the Neumann boundary condition is considered. Notice this assumption is just the Delaunay triangulation when $d=2$. In 3D, the notations in the assumption \eqref{eq20180907} are as follows: $a_i (1\leq i \leq d+1)$ denote the
vertices of $K$, $E=E_{ij}$ the edge connecting two vertices $a_i$ and
$a_j$, $F_i$ the $(d-1)$-dimensional simplex opposite to the vertex
$a_i$, $\theta_{ij}^K$ or $\theta_E^K$ the angle between the faces
$F_i$ and $F_j$, $\kappa_E^K=F_i \cap F_j$ , the $(d-2)$-dimensional simplex opposite to the edge $E=E_{ij}$. See Figure \ref{fig1} below.
\begin{figure}[H]
\centering 
\includegraphics[height=1.5in,width=2.0in]{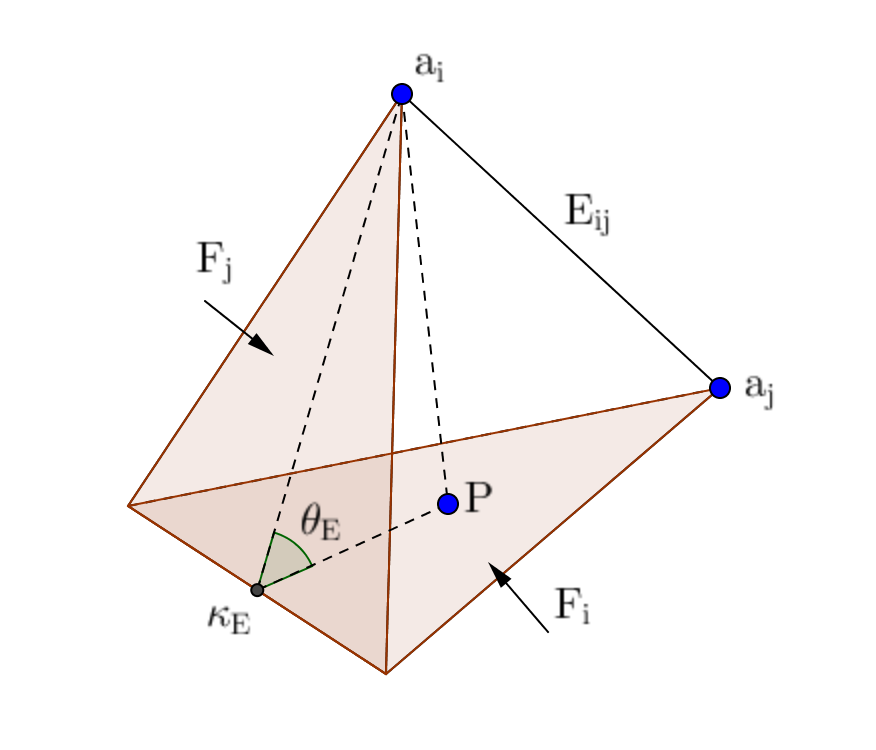} 
\caption{3D triangulation.}
\label{fig1}
\end{figure}

Consider the $\mathcal{P}_1$-Lagrangian finite element space
\begin{align}\label{eq20180713_1}
V_h = \bigl\{v_h \in H^1(\D): v_h|_{K} \in \mathcal{P}_1(K)\quad\forall K\in\mathcal{T}_h\bigr\},
\end{align}
where $\mathcal{P}_1$ denotes the space of all linear polynomials. Then the finite element approximation of \eqref{var_solu} is to seek an $\mathcal{F}_{t_n}$ adapted $V_h$-valued process $\{u_h^n\}_{n=1}^N$ such that it holds $\mathbb{P}$-almost surely that
\begin{align}\label{dfem}
(u^{n+1}_h, v_h) &+ \tau ( \nabla u^{n+1}_h, \nabla v_h )\\
&= (u^n_h, v_h) + \tau (I_hf^{n+1}, v_h)+ (g(u^n_h), v_h) \, \bar{\Delta} W_{n+1} \qquad \forall \, v_h \in V_h \notag,
\end{align}
where $f^{n+1}:=u^{n+1}_h-(u^{n+1}_h)^q$, $\bar{\Delta} W_{n+1}=W(t_{n+1})-W(t_n) \sim \mathcal{N} (0,\tau)$, and $I_h$ is the standard nodal value interpolation operator
$I_h: C(\bar{\Omega}) \longrightarrow V_h$, i.e.,
\begin{equation} \label{interpolation}
I_h v := \sum_{i=1}^{N_h} v(a_i)\varphi_i,
\end{equation} 
where $N_h$ denotes the number of vertices of $\mathcal{T}_h$, 
and ${\varphi_i}$ denotes the nodal basis function of $V_h$ corresponding to the vertex $a_i$.
The initial condition is chosen by $u_h^0  = P_h u_0$ where $P_h: L^2(\D) \longrightarrow V_h$ is the $L^2$-projection operator defined by
\begin{align*}
\bigl(P_h w, v_h\bigr) = (w, v_h) \qquad v_h \in V_h.
\end{align*}
\par 
For all $w \in H^s(\D)$, the following well-known error estimate results can be found in \cite{BS2008, ciarlet2002finite}:
\begin{align}
\label{Ph1}
&\|w - P_h w \|_{L^2} + h \| \nabla (w - P_h w) \|_{L^2} 
\leq C h^{\min\{2,s\}} \|w\|_{H^s},\\ 
\label{Ph2}
&\|w - P_h w\|_{L^\infty} \leq C h^{2-\frac{d}{2}} \|w\|_{H^2}. 
\end{align}

Finally, given $v_h \in {V}_h$, we define the discrete Laplace operator $\Delta_h: {V}_h\longrightarrow 
{V}_h$ by 
\begin{equation} \label{eq:discrete-Laplace}
(\Delta_h v_h, w_h)=-(\nabla v_h,\nabla w_h) \qquad \forall\, w_h\in V_h.
\end{equation}

\subsection{Stability estimates for the $p$-th moment of the $H^1$-seminorm of $u_h^n$} \label{sec-3.2}
First we shall prove the second moment discrete $H^1$-seminorm stability result, which is necessary to establish the corresponding higher moment stability result.
\begin{theorem}\label{thm20180711_1}
Suppose the mesh assumption in \eqref{eq20180907} holds, then
\begin{align}\label{eq20180711_7}
\sup_{0\leq n \leq N}\E\left[\|\nabla u^{n}_h\|_{L^2}^2\right]&+\frac14\sum_{n=0}^{N-1}\E\left[\|\nabla (u^{n+1}_h-u^{n}_h)\|_{L^2}^2\right] \\ 
&\quad +\tau\sum_{n=0}^{N-1}\E\left[\|\Delta_h u_h^{n+1}\|_{L^2}^2\right]\le C. \notag 
\end{align}
\end{theorem}
\begin{proof}
Testing \eqref{dfem} with $-\Delta_h u_h^{n+1}$, then
\begin{align}\label{eq20180711_1}
&(u^{n+1}_h-u^n_h, -\Delta_h u_h^{n+1}) + \tau ( \nabla u^{n+1}_h, -\nabla\Delta_h u_h^{n+1} )  \\
& \quad = \tau(I_hf^{n+1}, -\Delta_h u_h^{n+1})+ ( g(u^n_h), -\Delta_h u_h^{n+1}) \, \bar{\Delta} W_{n+1}.\notag
\end{align}

Using the definition of the discrete Laplace operator, we get
\begin{align}\label{eq20180711_2}
(u^{n+1}_h-u^n_h, -\Delta_h u_h^{n+1})&=\frac12\|\nabla u^{n+1}_h\|_{L^2}^2-\frac12\|\nabla u^{n}_h\|_{L^2}^2\\
&\qquad+\frac12\|\nabla (u^{n+1}_h-u^{n}_h)\|_{L^2}^2,\notag\\
\tau ( \nabla u^{n+1}_h, -\nabla\Delta_h u_h^{n+1} )&=\tau\|\Delta_h u_h^{n+1}\|_{L^2}^2,\label{eq20181009_2}\\
\E[( g(u^n_h), -\Delta_h u_h^{n+1}) \, \bar{\Delta} W_{n+1}]&=\E [(\nabla(P_hg(u^n_h)), \nabla (u_h^{n+1}-u^n_h)) \, \bar{\Delta} W_{n+1}]\label{eq20181009_3}\\
&\le C\tau\E[\|\nabla u^n_h\|_{L^2}^2]+\frac14\E[\|\nabla (u_h^{n+1}-u^n_h)\|_{L^2}^2],\notag
\end{align}
where the stability in the $H^1$-seminorm of the $L^2$ projection \cite{bank2014h} is used in the inequality of \eqref{eq20181009_3}.

The crucial part is to bound the first term on the right-hand side of \eqref{eq20180711_1} since it cannot be treated as a bad term, which aligns with the continuous case. Denote $u_i=u_h^{n+1}(a_i)$, then
\begin{align}\label{eq20180711_3}
\tau (I_hf^{n+1}, -\Delta_h u_h^{n+1})&=\tau\|\nabla u^{n+1}_h\|_{L^2}^2-\tau(\nabla\sum_{i=1}^{N_h} u_i^q \varphi_i,\nabla \sum_{j=1}^{N_h} u_j\varphi_j)\\
&=\tau\|\nabla u^{n+1}_h\|_{L^2}^2-\tau \sum_{i,j=1}^{N_h} ( u_i^q \nabla\varphi_i,  u_j\nabla\varphi_j)\notag\\
&=\tau\|\nabla u^{n+1}_h\|_{L^2}^2-\tau \sum_{i,j=1}^{N_h} b_{ij}(\nabla\varphi_i,  \nabla\varphi_j)\notag,
\end{align}
where $b_{ij}=u_i^q u_j$.

Using Young's inequality when $i\neq j$, we have
\begin{align}\label{eq20180711_4}
|b_{ij}|\le \frac{q}{q+1}u_i^{q+1}+\frac{1}{q+1}u_j^{q+1}.
\end{align}

Besides, since the stiffness matrix is diagonally dominant, then
\begin{align}
-\tau \sum_{i,j=1}^{N_h} b_{ij}(\nabla\varphi_i,  \nabla\varphi_j)&\le-\tau \sum_{k=1}^{N_h} b_{kk}[(\nabla\varphi_k,  \nabla\varphi_k)-\frac{q}{q+1}\sum_{i=1,\atop i\neq k}^{N_h} |(\nabla\varphi_i,  \nabla\varphi_k)|\\
&\quad-\frac{1}{q+1}\sum_{j=1,\atop j\neq k}^{N_h} |(\nabla\varphi_k,  \nabla\varphi_j)|]\notag\\
&\le-\tau \sum_{k=1}^{N_h} b_{kk}[(\nabla\varphi_k,  \nabla\varphi_k)-\sum_{i=1,\atop i\neq k}^{N_h} (\nabla\varphi_i,  \nabla\varphi_k)]\notag\\
&\le0\notag.
\end{align}

Then we have
\begin{align}\label{eq20180711_6}
\tau (I_hf^{n+1}, -\Delta_h u_h^{n+1})\le\tau\|\nabla u^{n+1}_h\|_{L^2}^2.
\end{align}

Combining \eqref{eq20180711_1}--\eqref{eq20181009_3} and \eqref{eq20180711_6}, and taking the summation, we have
\begin{align}\label{eq20180712_1}
&\frac12\E\left[\|\nabla u^{\ell}_h\|_{L^2}^2\right]+\frac14\sum_{n=0}^{\ell-1}\E\left[\|\nabla (u^{n+1}_h-u^{n}_h)\|_{L^2}^2\right]+\tau\sum_{n=0}^{\ell-1}\E\left[\|\Delta_h u_h^{n+1}\|_{L^2}^2\right]\\
&\quad\le C\tau\sum_{n=0}^{\ell-1}\E[\|\nabla u^n_h\|_{L^2}^2].\notag
\end{align}

Using Gronwall's inequality, we obtain \eqref{eq20180711_7}. 
\end{proof}

Before we establish the error estimates, we need to prove the stability of 
the higher order moments for the $H^1$-seminorm of the numerical solution. 

\begin{theorem}\label{thm20180802_1}
Suppose the mesh assumption in \eqref{eq20180907} holds, then for any $p\ge2$,
\begin{align*}
\sup_{0 \leq n \leq N} \E\left[\|\nabla u^{n}_h\|_{L^2}^p\right]\le C.
\end{align*}
\end{theorem}
\begin{proof}
The proof is divided into three steps. In Step 1, we establish the bound for $\E\|\nabla u^{\ell}_h\|_{L^2}^{4}$. 
In Step 2, we give the bound for $\E\|\nabla u^{\ell}_h\|_{L^2}^p$, where $p=2^r$ and $r$ is an arbitrary 
positive integer. In Step 3, we obtain the bound for $\E\|\nabla u^{\ell}_h\|_{L^2}^p$, where $p$ is an 
arbitrary real number and $p\ge2$.

\smallskip
{\bf Step 1.} Based on \eqref{eq20180711_1}--\eqref{eq20180711_6}, we have
\begin{align}\label{eq20180802_1}
&\frac12 \|\nabla u^{n+1}_h\|_{L^2}^2 -\frac12 \|\nabla u^{n}_h\|_{L^2}^2 +\frac12 \|\nabla (u^{n+1}_h-u^{n}_h)\|_{L^2}^2 +\tau \|\Delta_h u_h^{n+1}\|_{L^2}^2 \\
&\quad -( g(u^n_h), -\Delta_h u_h^{n+1}) \, \bar{\Delta} W_{n+1}\le\tau \|\nabla u^{n+1}_h\|_{L^2}^2.\notag
\end{align}

Notice the following identity
\begin{align}\label{eq20180802_2}
\|\nabla u^{n+1}_h\|_{L^2}^2+\frac12\|\nabla u^{n}_h\|_{L^2}^2=&\frac34(\|\nabla u^{n+1}_h\|_{L^2}^2+\|\nabla u^n_h\|_{L^2}^2)+\frac14(\|\nabla u^{n+1}_h\|_{L^2}^2-\|\nabla u^n_h\|_{L^2}^2),
\end{align}
and multiplying \eqref{eq20180802_1} with $\|\nabla u^{n+1}_h\|_{L^2}^2+\frac12\|\nabla u^{n}_h\|_{L^2}^2$, we obtain
\begin{align}\label{eq20180802_3}
&\frac38(\|\nabla u^{n+1}_h\|_{L^2}^4-\|\nabla u^n_h\|_{L^2}^4)+\frac18(\|\nabla u^{n+1}_h\|_{L^2}^2-\|\nabla u^n_h\|_{L^2}^2)^2\\
&\quad+(\frac12\|\nabla (u^{n+1}_h-u^{n}_h)\|_{L^2}^2+\tau\|\Delta_h u_h^{n+1}\|_{L^2}^2)(\|\nabla u^{n+1}_h\|_{L^2}^2+\frac12\|\nabla u^{n}_h\|_{L^2}^2)\notag\\
&\le\tau\|\nabla u^{n+1}_h\|_{L^2}^2(\|\nabla u^{n+1}_h\|_{L^2}^2+\frac12\|\nabla u^{n}_h\|_{L^2}^2)\notag\\
&\quad+( g(u^n_h), -\Delta_h u_h^{n+1}) \, \bar{\Delta} W_{n+1}(\|\nabla u^{n+1}_h\|_{L^2}^2+\frac12\|\nabla u^{n}_h\|_{L^2}^2).\notag
\end{align}

The first term on the right-hand side of \eqref{eq20180802_3} can be written as
\begin{align}\label{eq20180802_4}
&\tau\|\nabla u^{n+1}_h\|_{L^2}^2(\|\nabla u^{n+1}_h\|_{L^2}^2+\frac12\|\nabla u^{n}_h\|_{L^2}^2)\\
&\quad=\tau\|\nabla u^{n+1}_h\|_{L^2}^2(\frac32\|\nabla u^{n+1}_h\|_{L^2}^2-\frac12(\|\nabla u^{n+1}_h\|_{L^2}^2-\|\nabla u^n_h\|_{L^2}^2))\notag\\
&\quad\le C\tau\|\nabla u^{n+1}_h\|_{L^2}^4+\theta_1(\|\nabla u^{n+1}_h\|_{L^2}^2-\|\nabla u^n_h\|_{L^2}^2)^2\notag,
\end{align}
where $\theta_1>0$ will be determined later.

The second term on the right-hand side of \eqref{eq20180802_3} can be written as
\begin{align}\label{eq20180802_5}
&(g(u^n_h), -\Delta_h u_h^{n+1}) \, \bar{\Delta} W_{n+1}(\|\nabla u^{n+1}_h\|_{L^2}^2+\frac12\|\nabla u^{n}_h\|_{L^2}^2)\\
&\quad=(\nabla P_hg(u^n_h), \nabla u_h^{n+1}) \, \bar{\Delta} W_{n+1}(\|\nabla u^{n+1}_h\|_{L^2}^2+\frac12\|\nabla u^{n}_h\|_{L^2}^2)\notag\\
&\quad= ((\nabla P_hg(u^n_h), \nabla u_h^{n+1}-\nabla u_h^n)\bar{\Delta} W_{n+1}\notag\\
&\qquad+(\nabla P_hg(u^n_h),\nabla u_h^n)\bar{\Delta} W_{n+1})(\|\nabla u^{n+1}_h\|_{L^2}^2+\frac12\|\nabla u^{n}_h\|_{L^2}^2)\notag\\
&\quad\le(\frac14\|\nabla u_h^{n+1}-\nabla u_h^n\|_{L^2}^2+C\|\nabla u_h^n\|_{L^2}^2(\bar{\Delta} W_{n+1})^2\notag\\
&\qquad+(\nabla P_hg(u^n_h),\nabla u_h^n)\bar{\Delta} W_{n+1})(\|\nabla u^{n+1}_h\|_{L^2}^2+\frac12\|\nabla u^{n}_h\|_{L^2}^2)\notag.
\end{align}

For the right-hand side of \eqref{eq20180802_5}, using the Cauchy-Schwarz inequality, we get
\begin{align}\label{eq20180802_6}
&C\|\nabla u_h^n\|_{L^2}^2(\bar{\Delta} W_{n+1})^2(\|\nabla u^{n+1}_h\|_{L^2}^2+\frac12\|\nabla u^{n}_h\|_{L^2}^2)\\
&\quad=C\|\nabla u_h^n\|_{L^2}^2(\bar{\Delta} W_{n+1})^2(\|\nabla u^{n+1}_h\|_{L^2}^2-\|\nabla u^{n}_h\|_{L^2}^2+\frac32\|\nabla u^{n}_h\|_{L^2}^2)\notag\\
&\quad\le\theta_2(\|\nabla u^{n+1}_h\|_{L^2}^2-\|\nabla u^n_h\|_{L^2}^2)^2+C\|\nabla u_h^n\|_{L^2}^4(\bar{\Delta} W_{n+1})^4\notag\\
&\qquad+C\|\nabla u_h^n\|_{L^2}^4(\bar{\Delta} W_{n+1})^2,\notag
\end{align}
where $\theta_2>0$ will be determined later. 
Similarly, using the Cauchy-Schwarz inequality, we have
\begin{align}\label{eq20180802_7}
&(\nabla P_hg(u^n_h),\nabla u_h^n)\bar{\Delta} W_{n+1}(\|\nabla u^{n+1}_h\|_{L^2}^2+\frac12\|\nabla u^{n}_h\|_{L^2}^2)\\
&\quad=(\nabla P_hg(u^n_h),\nabla u_h^n)\bar{\Delta} W_{n+1}(\|\nabla u^{n+1}_h\|_{L^2}^2-\|\nabla u^{n}_h\|_{L^2}^2+\frac32\|\nabla u^{n}_h\|_{L^2}^2)\notag\\
&\quad\le\theta_3(\|\nabla u^{n+1}_h\|_{L^2}^2-\|\nabla u^n_h\|_{L^2}^2)^2+C\|\nabla u_h^n\|_{L^2}^4(\bar{\Delta} W_{n+1})^2\notag\\
&\qquad+\frac32(\nabla P_hg(u^n_h),\nabla u_h^n)\bar{\Delta} W_{n+1}\|\nabla u^{n}_h\|_{L^2}^2\notag,
\end{align}
where $\theta_3>0$ will be determined later.

Choosing $\theta_1, \theta_2, \theta_3$ such that $\theta_1+\theta_2+\theta_3\le\frac{1}{16}$, then taking the summation over $n$ from $0$ to $\ell-1$ and taking the expectation on both sides of \eqref{eq20180802_3}, we obtain
\begin{align}\label{eq20180802_8}
&\frac38\E\left[\|\nabla u^{\ell}_h\|_{L^2}^4\right]+\frac{1}{16}\sum_{n=0}^{\ell-1}\E\left[(\|\nabla u^{n+1}_h\|_{L^2}^2-\|\nabla u^n_h\|_{L^2}^2)^2\right]\\
&\quad+\sum_{n=0}^{\ell-1}\E\left[(\frac14\|\nabla (u^{n+1}_h-u^{n}_h)\|_{L^2}^2+\tau\|\Delta_h u_h^{n+1}\|_{L^2}^2)(\|\nabla u^{n+1}_h\|_{L^2}^2+\frac12\|\nabla u^{n}_h\|_{L^2}^2)\right]\notag\\
&\le C\tau\sum_{n=0}^{\ell-1}\E\left[\|\nabla u^{n+1}_h\|_{L^2}^4\right]+\frac38\E\left[\|\nabla u^0_h\|_{L^2}^4\right]+C\tau^2\sum_{n=0}^{\ell-1}\E\left[\|\nabla u_h^n\|_{L^2}^4\right]\notag\\
&\quad+C\tau\sum_{n=0}^{\ell-1}\E\left[\|\nabla u_h^n\|_{L^2}^4\right].\notag
\end{align}

When restricting $\tau\le C$, we have
\begin{align}\label{eq20180802_9}
&\frac14\E\left[\|\nabla u^{\ell}_h\|_{L^2}^4\right]+\frac{1}{16}\sum_{n=0}^{\ell-1}\E\left[(\|\nabla u^{n+1}_h\|_{L^2}^2-\|\nabla u^n_h\|_{L^2}^2)^2\right]\\
&\quad+\sum_{n=0}^{\ell-1}\E\left[(\frac14\|\nabla (u^{n+1}_h-u^{n}_h)\|_{L^2}^2+\tau\|\Delta_h u_h^{n+1}\|_{L^2}^2)(\|\nabla u^{n+1}_h\|_{L^2}^2+\frac12\|\nabla u^{n}_h\|_{L^2}^2)\right]\notag\\
&\le C\tau\sum_{n=0}^{\ell-1}\E\left[\|\nabla u^{n}_h\|_{L^2}^4\right]+\frac38\E\left[\|\nabla u^0_h\|_{L^2}^4\right].\notag
\end{align}

Using Gronwall's inequality, we obtain
\begin{align}\label{eq20180802_10}
&\frac14\E\left[\|\nabla u^{\ell}_h\|_{L^2}^4\right]+\frac{1}{16}\sum_{n=0}^{\ell-1}\E\left[(\|\nabla u^{n+1}_h\|_{L^2}^2-\|\nabla u^n_h\|_{L^2}^2)^2\right]\\
&\qquad+\sum_{n=0}^{\ell-1}\E\bigl[(\frac14\|\nabla (u^{n+1}_h-u^{n}_h)\|_{L^2}^2+\tau\|\Delta_h u_h^{n+1}\|_{L^2}^2)(\|\nabla u^{n+1}_h\|_{L^2}^2\notag\\
&\qquad+\frac12\|\nabla u^{n}_h\|_{L^2}^2)\bigr]\le C.\notag
\end{align}

\smallskip
{\bf Step 2.} Similar to Step 1, using \eqref{eq20180802_3}--\eqref{eq20180802_7}, we have
\begin{align}\label{eq20180802_11}
&\frac38(\|\nabla u^{n+1}_h\|_{L^2}^4-\|\nabla u^n_h\|_{L^2}^4)+\frac{1}{16}(\|\nabla u^{n+1}_h\|_{L^2}^2-\|\nabla u^n_h\|_{L^2}^2)^2\\
&\quad+(\frac14\|\nabla (u^{n+1}_h-u^{n}_h)\|_{L^2}^2+\tau\|\Delta_h u_h^{n+1}\|_{L^2}^2)(\|\nabla u^{n+1}_h\|_{L^2}^2+\frac12\|\nabla u^{n}_h\|_{L^2}^2)\notag\\
&\le C\tau\|\nabla u^{n+1}_h\|_{L^2}^4+C\|\nabla u_h^n\|_{L^2}^4(\bar{\Delta} W_{n+1})^4+C\|\nabla u_h^n\|_{L^2}^4(\bar{\Delta} W_{n+1})^2\notag\\
&\quad+C\|\nabla u_h^n\|_{L^2}^4\bar{\Delta} W_{n+1}.\notag
\end{align}

Proceed similarly as in Step 1, multiplying \eqref{eq20180802_11} with $\|\nabla u^{n+1}_h\|_{L^2}^4+\frac12\|\nabla u^{n}_h\|_{L^2}^4$, we can obtain the 8-th moment of the $H^1$-seminorm stability result of the numerical solution. 
Then repeating this process, the $2^r$-th moment of the $H^1$-seminorm stability result of the numerical solution 
can be obtained.

\smallskip
{\bf Step 3.} Suppose $2^{r-1}\le p\le 2^r$, then using Young's inequality, we have
\begin{align}\label{eq20180802_12}
\E\left[\|\nabla u^{\ell}_h\|_{L^2}^p\right]&\le \E\left[\|\nabla u^{\ell}_h\|_{L^2}^{2^r}\right]+C< \infty,
\end{align}
where the second inequality follows from the results of Step 2. The proof is complete.
\end{proof}

\subsection{Stability estimates for the $p$-th moment of the $L^2$-norm of $u_h^n$}
Since the mass matrix may not be the diagonally dominated matrix, we cannot use the above idea to prove the $L^2$ stability. Instead, we prove the stability results by utilizing the above established results. The following results hold when $q\ge3$ is the odd integer in 2D case, and when $q=3$ or $q=5$ in 3D case.
\begin{theorem}\label{thm20180911}
Suppose the mesh assumption in \eqref{eq20180907} holds, then
\begin{align*}
&\sup_{0\leq n \leq N}\E\left[\|u^{n}_h\|_{L^2}^2\right]+\sum_{n=0}^{N-1}\E\left[\|(u^{n+1}_h-u^{n}_h)\|_{L^2}^2\right]+\tau\sum_{n=0}^{N-1}\E\left[\|\nabla u_h^{n+1}\|_{L^2}^2\right]\\
&\qquad+\frac{\tau}{2}\sum_{n=0}^{N-1}\E\left[\|u_h^{n+1}\|_{L^{q+1}}^{q+1}\right]\le C.
\end{align*}
\end{theorem}
\begin{proof}
Testing \eqref{dfem} with $u_h^{n+1}$, then
\begin{align}\label{eq20180711_10}
&(u^{n+1}_h-u^n_h, u_h^{n+1}) + \tau ( \nabla u^{n+1}_h, \nabla u_h^{n+1} )\\
&\quad= \tau(I_hf^{n+1}, u_h^{n+1})+(g(u^n_h), u_h^{n+1}) \, \bar{\Delta} W_{n+1}\notag.
\end{align}

We can easily prove the following inequalities:
\begin{align*}
(u^{n+1}_h-u^n_h, u_h^{n+1})&=\frac12\| u^{n+1}_h\|_{L^2}^2-\frac12\| u^{n}_h\|_{L^2}^2+\frac12\|u^{n+1}_h-u^{n}_h\|_{L^2}^2,\\
\E[(g(u^n_h),  u_h^{n+1}) \, \bar{\Delta} W_{n+1}]&=\E [(g(u^n_h),  (u_h^{n+1}-u^n_h)) \, \bar{\Delta} W_{n+1}]\\
&\le C\tau+C\tau\E[\| u^n_h\|_{L^2}^2]+\frac14\E[\|u_h^{n+1}-u^n_h\|_{L^2}^2],
\end{align*}
where \eqref{eq20180813_1} is used in the inequality above.

We have the following standard interpolation result and the inverse inequality \cite{ciarlet2002finite}:
\begin{align}\label{eq20180807_7}
\|v-I_hv\|_{L^{\frac{q+1}{q}}(K)}\le Ch_K\|\nabla v\|_{L^{\frac{q+1}{q}}(K)},\\
\|v\|_{L^{q+1}(K)}^{q+1}\le\frac{C}{h_K^{d\cdot\frac{q-1}{2}}}\|v\|_{L^2(K)}^{q+1}.\label{eq20180807_8}
\end{align}

Using \eqref{eq20180807_7}--\eqref{eq20180807_8}, and Young's inequality, we have
\begin{align}\label{eq20180808_2}
&\tau(I_hf^{n+1}, u_h^{n+1})=\tau(f^{n+1}, u_h^{n+1})-\tau(f^{n+1}-I_hf^{n+1}, u_h^{n+1})\\
&\quad\le\tau\|u_h^{n+1}\|_{L^2}^2-\tau\|u_h^{n+1}\|_{L^{q+1}}^{q+1}\notag\\
&\qquad+C\tau\|f^{n+1}-I_hf^{n+1}\|_{L^{\frac{q+1}{q}}}^{\frac{q+1}{q}}+\frac{\tau}{4}\|u_h^{n+1}\|_{L^{q+1}}^{q+1}\notag\\
&\quad\le\tau\|u_h^{n+1}\|_{L^2}^2-\tau\|u_h^{n+1}\|_{L^{q+1}}^{q+1}\notag\\
&\qquad+C\tau \sum_{K\in\mathcal{T}_h}h_K^{\frac{q+1}{q}}\bigl((u_h^{n+1})^{\frac{q^2-1}{q}},(\nabla u_h^{n+1})^{\frac{q+1}{q}}\bigr)_K+\frac{\tau}{4}\|u_h^{n+1}\|_{L^{q+1}}^{q+1}\notag\\
&\quad\le\tau\|u_h^{n+1}\|_{L^2}^2-\frac{\tau}{2}\|u_h^{n+1}\|_{L^{q+1}}^{q+1}+C\tau \sum_{K\in\mathcal{T}_h}h_K^{q+1}\|\nabla u_h^{n+1}\|_{L^{q+1}(K)}^{q+1}\notag\\
&\quad\le\tau\|u_h^{n+1}\|_{L^2}^2-\frac{\tau}{2}\|u_h^{n+1}\|_{L^{q+1}}^{q+1}+C\tau\sum_{K\in\mathcal{T}_h}h_K^{q+1-d\frac{q-1}{2}}\|\nabla u_h^{n+1}\|_{L^2(K)}^{q+1}.\notag
\end{align}

Notice when $d=2$, $q+1-d\frac{q-1}{2}\ge0$ if $q\ge0$, and when $d=3$, $q+1-d\frac{q-1}{2}\ge0$ if $q\le5$. Using the above inequalities, Theorem \ref{thm20180802_1}, taking summation over $n$ from $0$ to $\ell-1$, and taking expectation on both sides of \eqref{eq20180711_10}, we obtain
\begin{align}\label{eq20180807_1}
&\frac14\E\left[\|u^{\ell}_h\|_{L^2}^2\right]+\frac14\sum_{n=0}^{\ell-1}\E\left[\|(u^{n+1}_h-u^{n}_h)\|_{L^2}^2\right]+\tau\sum_{n=0}^{\ell-1}\E\left[\|\nabla u_h^{n+1}\|_{L^2}^2\right]\\
&\qquad+\frac{\tau}{2}\sum_{n=0}^{\ell-1}\E\left[\|u_h^{n+1}\|_{L^{q+1}}^{q+1}\right]\notag\\
&\le\tau\sum_{n=0}^{\ell-1}\E\left[\|u^n_h\|_{L^2}^2\right]+C\tau\sum_{n=0}^{\ell-1}\E\left[\|\nabla u_h^{n+1}\|_{L^2}^{q+1}\right]+C\notag\\
&\le\tau\sum_{n=0}^{\ell-1}\E\left[\|u^n_h\|_{L^2}^2\right]+C,\notag
\end{align}
where Theorem \ref{thm20180802_1} is used in the last inequality.

The conclusion is a direct result by using Gronwall's inequality.
\end{proof}

%
%

To obtain the error estimates results, we need to establish a higher moment discrete $L^2$ stability result
for the numerical solution $u_h$.

\begin{theorem}\label{thm20180808_1}
Suppose the mesh assumption in \eqref{eq20180907} holds, then for any $p\ge2$,
\begin{align*}
\sup_{0\leq \ell \leq N}\E\left[\|u^{\ell}_h\|_{L^2}^p\right]\le C\notag.
\end{align*}
\end{theorem}

\begin{proof}
The proof is divided into three steps. In Step 1, we give the bound for $\E\| u^{\ell}_h\|_{L^2}^{4}$. 
In Step 2, we give the bound for $\E\| u^{\ell}_h\|_{L^2}^p$, where $p=2^r$ and $r$ is an arbitrary positive integer.
In Step 3, we give the bound for $\E\| u^{\ell}_h\|_{L^2}^p$, where $p$ is an arbitrary real number and $p\ge2$.

\smallskip
{\it Step 1.} Based on \eqref{eq20180711_10} and  \eqref{eq20180808_2}, we have
\begin{align}\label{eq20180808_3}
&\frac12\|u^{n+1}_h\|_{L^2}^2-\frac12\|u^{n}_h\|_{L^2}^2+\frac12\|u^{n+1}_h-u^{n}_h\|_{L^2}^2+\tau\|\nabla u_h^{n+1}\|_{L^2}^2+\frac{\tau}{2}\|u_h^{n+1}\|_{L^{q+1}}^{q+1}\\
&\qquad \le\tau\|u_h^{n+1}\|_{L^2}^2+C\tau\|\nabla u_h^{n+1}\|_{L^2}^{q+1}+(g(u^n_h), u_h^{n+1}) \, \bar{\Delta} W_{n+1}.\notag
\end{align}

Notice the following identity
\begin{align}\label{eq20180808_4}
\|u^{n+1}_h\|_{L^2}^2+\frac12\|u^{n}_h\|_{L^2}^2=&\frac34(\|u^{n+1}_h\|_{L^2}^2+\|u^n_h\|_{L^2}^2)+\frac14(\|u^{n+1}_h\|_{L^2}^2-\|u^n_h\|_{L^2}^2).
\end{align}
Multiplying \eqref{eq20180808_3} with $\|u^{n+1}_h\|_{L^2}^2+\frac12\|u^{n}_h\|_{L^2}^2$, we obtain
\begin{align}\label{eq20180808_5}
&\frac38(\|u^{n+1}_h\|_{L^2}^4-\|u^n_h\|_{L^2}^4)+\frac18(\|u^{n+1}_h\|_{L^2}^2-\|u^n_h\|_{L^2}^2)^2+(\frac12\|(u^{n+1}_h-u^{n}_h)\|_{L^2}^2\\
&\quad+\tau\|\nabla u_h^{n+1}\|_{L^2}^2+\frac{\tau}{2}\|u_h^{n+1}\|_{L^{q+1}}^{q+1})(\|u^{n+1}_h\|_{L^2}^2+\frac12\|u^{n}_h\|_{L^2}^2)\notag\\
&\le(\tau\|u_h^{n+1}\|_{L^2}^2+C\tau\|\nabla u_h^{n+1}\|_{L^2}^{q+1})(\|u^{n+1}_h\|_{L^2}^2+\frac12\|u^{n}_h\|_{L^2}^2)\notag\\
&\quad+(g(u^n_h), u_h^{n+1})\, \bar{\Delta} W_{n+1}(\|u^{n+1}_h\|_{L^2}^2+\frac12\|u^{n}_h\|_{L^2}^2).\notag
\end{align}

The first term on the right-hand side of \eqref{eq20180808_5} can be written as
\begin{align}\label{eq20180808_6}
&(\tau\|u_h^{n+1}\|_{L^2}^2+C\tau\|\nabla u_h^{n+1}\|_{L^2}^{q+1})(\|u^{n+1}_h\|_{L^2}^2+\frac12\|u^{n}_h\|_{L^2}^2)\\
&\le\tau\|u^{n+1}_h\|_{L^2}^2(\frac32\|u^{n+1}_h\|_{L^2}^2-\frac12(\|u^{n+1}_h\|_{L^2}^2-\|u^n_h\|_{L^2}^2))\notag\\
&\quad+C\tau\|\nabla u_h^{n+1}\|_{L^2}^{2(q+1)}+\tau\|u^{n+1}_h\|_{L^2}^4+\tau(\|u^{n+1}_h\|_{L^2}^2-\|u^n_h\|_{L^2}^2)^2\notag\\
&\le C\tau\|u^{n+1}_h\|_{L^2}^4+C\tau\|\nabla u_h^{n+1}\|_{L^2}^{2(q+1)}+\theta_1(\|u^{n+1}_h\|_{L^2}^2-\|u^n_h\|_{L^2}^2)^2\notag,
\end{align}
where $\theta_1>0$ will be determined later.

The second term on the right-hand side of \eqref{eq20180808_5} can be written as
\begin{align}\label{eq20180808_7}
&(g(u^n_h), u_h^{n+1})\, \bar{\Delta} W_{n+1}(\|u^{n+1}_h\|_{L^2}^2+\frac12\|u^{n}_h\|_{L^2}^2)\\
&=(g(u^n_h), u_h^{n+1}-u_h^n+u_h^n) \, \bar{\Delta} W_{n+1}(\|u^{n+1}_h\|_{L^2}^2+\frac12\|u^{n}_h\|_{L^2}^2)\notag\\
&\le(\frac14\|u_h^{n+1}-u_h^n\|_{L^2}^2+C(1+\|u_h^n\|_{L^2}^2)(\bar{\Delta} W_{n+1})^2\notag\\
&\quad+(g(u_h^n),u_h^n) \bar{\Delta} W_{n+1})(\|u^{n+1}_h\|_{L^2}^2+\frac12\|u^{n}_h\|_{L^2}^2)\notag.
\end{align}
For the second term on the right-hand side of \eqref{eq20180808_7}, using the Cauchy-Schwarz inequality, we get
\begin{align}\label{eq20180808_8}
&C(1+\|u_h^n\|_{L^2}^2)(\bar{\Delta} W_{n+1})^2(\|u^{n+1}_h\|_{L^2}^2+\frac12\|u^{n}_h\|_{L^2}^2)\\
&=C(1+\|u_h^n\|_{L^2}^2)(\bar{\Delta} W_{n+1})^2(\|u^{n+1}_h\|_{L^2}^2-\| u^{n}_h\|_{L^2}^2+\frac32\|u^{n}_h\|_{L^2}^2)\notag\\
&\le\theta_2\big(\|u^{n+1}_h\|_{L^2}^2-\|u^n_h\|_{L^2}^2)^2+(C+C\|u_h^n\|_{L^2}^4)(\bar{\Delta} W_{n+1})^4\notag\\
&\quad+C\|u_h^n\|_{L^2}^4(\bar{\Delta} W_{n+1} \big)^2+C\|u_h^n\|_{L^2}^2(\bar{\Delta} W_{n+1})^2,\notag
\end{align}
where $\theta_2>0$ will be determined later. 
Using \eqref{eq20180813_2}, the third term on the right-hand side of \eqref{eq20180808_7} can be bounded by
\begin{align}\label{eq20180808_9}
&(g(u_h^n),u_h^n) \bar{\Delta} W_{n+1}(\|u^{n+1}_h\|_{L^2}^2+\frac12\|u^{n}_h\|_{L^2}^2)\\
&\quad=(g(u_h^n),u_h^n) \bar{\Delta} W_{n+1}(\|u^{n+1}_h\|_{L^2}^2-\|u^{n}_h\|_{L^2}^2+\frac32\|u^{n}_h\|_{L^2}^2)\notag\\
&\quad\le\theta_3(\|u^{n+1}_h\|_{L^2}^2-\|u^n_h\|_{L^2}^2)^2+(C+C\|u_h^n\|_{L^2}^4)(\bar{\Delta} W_{n+1})^2\notag\\
&\qquad+\frac32(g(u_h^n),u_h^n)\|u_h^n\|_{L^2}^2\bar{\Delta} W_{n+1}\notag,
\end{align}
where $\theta_3>0$ will be determined later.

Choosing $\theta_1, \theta_2, \theta_3$ such that $\theta_1+\theta_2+\theta_3\le\frac{1}{16}$, then taking the summation over $n$ from $0$ to $\ell-1$ and taking the expectation on both sides of \eqref{eq20180808_5}, we obtain
\begin{align}\label{eq20180808_10}
&\frac38\E\left[\|u^{\ell}_h\|_{L^2}^4\right]+\frac{1}{16}\sum_{n=0}^{\ell-1}\E\left[(\|u^{n+1}_h\|_{L^2}^2-\|u^n_h\|_{L^2}^2)^2\right]+\sum_{n=0}^{\ell-1}\E\bigl[(\frac14\|(u^{n+1}_h-u^{n}_h)\|_{L^2}^2\\
&\quad+\tau\|\nabla u_h^{n+1}\|_{L^2}^2+\frac{\tau}{2}\|u_h^{n+1}\|_{L^{q+1}}^{q+1})(\|u^{n+1}_h\|_{L^2}^2+\frac12\|u^{n}_h\|_{L^2}^2)\bigr]\notag\\
&\le C\tau\sum_{n=0}^{\ell-1}\E\left[\|u^{n+1}_h\|_{L^2}^4\right]+C\tau\sum_{n=0}^{\ell-1}\E\left[\|\nabla u^{n+1}_h\|_{L^2}^{2(q+1)}\right]+\frac38\E\left[\|u^0_h\|_{L^2}^4\right]\notag\\
&\quad+C\tau\sum_{n=0}^{\ell-1}\E\left[\|u_h^n\|_{L^2}^4\right]+C.\notag
\end{align}

When $\tau\le C$, we have
\begin{align}\label{eq20180808_11}
&\frac14\E\left[\|u^{\ell}_h\|_{L^2}^4\right]+\frac{1}{16}\sum_{n=0}^{\ell-1}\E\left[(\|u^{n+1}_h\|_{L^2}^2-\|u^n_h\|_{L^2}^2)^2\right]+\sum_{n=0}^{\ell-1}\E\bigl[(\frac14\|(u^{n+1}_h-u^{n}_h)\|_{L^2}^2\\
&\quad+\tau\|\nabla u_h^{n+1}\|_{L^2}^2+\frac{\tau}{2}\|u_h^{n+1}\|_{L^4}^4)(\|u^{n+1}_h\|_{L^2}^2+\frac12\|u^{n}_h\|_{L^2}^2)\bigr]\notag\\
&\le C\tau\sum_{n=0}^{\ell-1}\E\left[\|u^{n}_h\|_{L^2}^4\right]+C\tau\sum_{n=0}^{\ell-1}\E\left[\|\nabla u^{n+1}_h\|_{L^2}^{2(q+1)}\right]+\frac38\E\left[\|u^0_h\|_{L^2}^4\right]+C\notag.
\end{align}

Using Gronwall's inequality, we obtain
\begin{align}\label{eq20180808_12}
&\frac14\E\left[\|u^{\ell}_h\|_{L^2}^4\right]+\frac{1}{16}\sum_{n=0}^{\ell-1}\E\left[(\|u^{n+1}_h\|_{L^2}^2-\|u^n_h\|_{L^2}^2)^2\right]+\sum_{n=0}^{\ell-1}\E\bigg[(\frac14\|(u^{n+1}_h-u^{n}_h)\|_{L^2}^2\\
&\quad+\tau\|\nabla u_h^{n+1}\|_{L^2}^2+\frac{\tau}{2}\|u_h^{n+1}\|_{L^4}^4)(\|u^{n+1}_h\|_{L^2}^2+\frac12\|u^{n}_h\|_{L^2}^2)\bigg]\le C.\notag
\end{align}

\smallskip
{\it Step 2.} Similar to Step 1, using \eqref{eq20180808_5}--\eqref{eq20180808_9}, we have
\begin{align}\label{eq20180808_13}
&\frac38(\|u^{n+1}_h\|_{L^2}^4-\|u^n_h\|_{L^2}^4)+\frac{1}{16}(\|u^{n+1}_h\|_{L^2}^2-\|u^n_h\|_{L^2}^2)^2\\
&+(\frac14\|(u^{n+1}_h-u^{n}_h)\|_{L^2}^2+\tau\|\nabla u_h^{n+1}\|_{L^2}^2+\frac{\tau}{2}\|u_h^{n+1}\|_{L^4}^4)(\|u^{n+1}_h\|_{L^2}^2+\frac12\|u^{n}_h\|_{L^2}^2)\notag\\
&\le C\tau\|u^{n+1}_h\|_{L^2}^4+C\tau\|\nabla u_h^{n+1}\|_{L^2}^{2(q+1)}+(C+C\|u_h^n\|_{L^2}^4)(\bar{\Delta} W_{n+1})^4\notag\\
&+(C+C\|u_h^n\|_{L^2}^4)(\bar{\Delta} W_{n+1})^2+(g(u_h^n),u_h^n)\|u_h^n\|_{L^2}^2\bar{\Delta} W_{n+1}.\notag
\end{align}
Similar to Step 1, multiplying \eqref{eq20180808_13} with $\|u^{n+1}_h\|_{L^2}^4+\frac12\|u^{n}_h\|_{L^2}^4$, 
we can obtain the 8-th moment of the $L^2$ stability result of the discrete solution. Then repeating this 
process, the $2^r$-th moment of the $L^2$ stability result of the discrete solution can be obtained.

\smallskip
{\it Step 3.} Suppose $2^{r-1}\le p\le 2^r$, then using Young's inequality, we have
\begin{align}\label{eq20180808_14}
\E\left[\|u^{\ell}_h\|_{L^2}^p\right]&\le \E\left[\|u^{\ell}_h\|_{L^2}^{2^r}\right]+C\\
&\le C\notag,
\end{align}
where the second inequality uses Step 2. The proof is complete. 
\end{proof}

\subsection{Error estimates}\label{subsec3}
Let $e^n=u(t_n)-u_h^n$ $(n = 0,1,2,\ldots,N)$. 
In the following theorem, the $L^2$ projection is used in the proof of the error estimates and 
the strong convergence rate is given. 

\begin{theorem}\label{thm:derrest} 
Let $u$ and $\{ u_h^n \}_{n=1}^N$ denote respectively the solutions of problem \eqref{var_solu} and  scheme \eqref{dfem}, then there holds
\begin{align*}
&\sup_{0 \leq n \leq N} \E \left[\| e^n \|^2_{L^2} \right] 
+ \E \left[\tau\sum_{n=1}^N  \|\nabla e^n\|^2_{L^2}  \right]\le C\tau +Ch^2|\ln h|^2.
\end{align*} 
\end{theorem}
\begin{proof}
We write $e^n = \eta^n + \xi^n$ where
\begin{align*}
\eta^n := u(t_n) - P_h u(t_n) \quad \text{and} \quad \xi^n := P_h u(t_n) - u_h^n,  \quad n = 0,1,2,...,N.
\end{align*}
It follows from \eqref{var_solu} that for all $t_n$ ($n \geq 0$) there holds $\P$-almost surely
\begin{align}
\label{derrest:1}
&\bigl(u(t_{n+1}), v_h) - (u(t_n), v_h\bigr)+\int_{t_n}^{t_{n+1}} \bigl(\nabla u(s), \nabla v_h\bigr) \, ds \\
&\quad  =\int_{t_n}^{t_{n+1}} \bigl(f(u(s)), v_h\bigr) \, ds+\int_{t_n}^{t_{n+1}} \bigl( g(u(s)), v_h\bigr) \, dW(s) 
\qquad\qquad  \forall \, v_h \in V_h. \notag
\end{align}

Subtracting \eqref{dfem} from \eqref{derrest:1} and setting $v_h = \xi^{n+1}$, the following error equation holds $\P$-almost surely,
\begin{align} \label{derrest:3}
(\xi^{n+1} - \xi^n, \xi^{n+1}) &= -(\eta^{n+1} - \eta^n, \xi^{n+1}) 
- \int_{t_n}^{t_{n+1}} \bigl(\nabla u(s) - \nabla u_h^{n+1}, \nabla \xi^{n+1} \bigr) \, ds \\
&\qquad +\int_{t_n}^{t_{n+1}} \bigl(f(u(s)) - I_hf^{n+1}, \xi^{n+1} \bigr) \, ds \notag \\
   &\qquad +  \int_{t_n}^{t_{n+1}} \bigl(( g(u(s)) -  g(u_h^n)) , \xi^{n+1} \bigr) \, dW(s), \notag \\
   & := T_1 + T_2 + T_3 + T_4. \notag 
\end{align}

The left-hand side of \eqref{derrest:3} can be handled by
\begin{align} \label{derrest:4}
 \E \bigl[(\xi^{n+1} - \xi^n, \xi^{n+1}) \bigr] 
&= \frac{1}{2} \E \bigl[ \|\xi^{n+1}\|_{L^2}^2 - \|\xi^{n}\|_{L^2}^2 \bigr] \\
&\quad +  \frac{1}{2} \E \bigl[ \| \xi^{n+1} - \xi^n \|^2_{L^2} \bigr]. \notag
\end{align}

Next, we bound the right-hand side of \eqref{derrest:3}. First, since $P_h$ is the 
$L^2$-projection operator, we have $\E \left[ T_1 \right] = 0$.

For the second term on the right-hand side of \eqref{derrest:3}, using the H\"{o}lder continuity in Lemma \ref{lem:e3}, we have
\begin{align} \label{derrest:6}
\E \left[ T_2 \right] &= - \E \left[\int_{t_n}^{t_{n+1}} (\nabla u(s) - \nabla u(t_{n+1}), \nabla \xi^{n+1}) \, ds \right] \\
&\quad 
- \E\left[\int_{t_n}^{t_{n+1}} (\nabla \eta^{n+1} + \nabla \xi^{n+1}, \nabla \xi^{n+1}) \, ds \right] \notag \\
&\leq C \E \left[\int_{t_n}^{t_{n+1}} \| \nabla u(s) - \nabla u(t_{n+1}) \|_{L^2}^2 \, ds \right] 
 \notag \\
&\quad - \frac{3}{4} \E \left[ \| \nabla \xi^{n+1} \|^2_{L^2} \right] \tau+ C\E \left[\| \nabla \eta^{n+1} \|^2_{L^2} \right] \tau \notag\\
&\leq C \tau^2 + C\E \left[ \| \nabla \eta^{n+1} \|^2_{L^2} \right] \tau 
- \frac{3}{4} \E \left[\| \nabla \xi^{n+1} \|^2_{L^2} \right] \tau. \notag
\end{align}

In order to estimate the third term on the right-hand side of \eqref{derrest:3}, we write 
\begin{align}
\label{derrest:7}
\bigl( f(u(s)) - I_hf^{n+1}, \xi^{n+1} \bigr) &= \bigl( f(u(s)) - f(u(t_{n+1})), \xi^{n+1} \bigr) \\
   &\quad + \bigl( f(u(t_{n+1}) - f(P_h u(t_{n+1})), \xi^{n+1} \bigr) \notag \\
   &\quad +  \bigl( f(P_h u(t_{n+1})) - f^{n+1}, \xi^{n+1} \bigr) \notag\\
    &\quad +  \bigl( f^{n+1}-I_hf^{n+1}, \xi^{n+1} \bigr) \notag.
\end{align}

Using the H\"{o}lder continuity in Lemma \ref{lem:e4}, we obtain
\begin{align}
\label{derrest:8}
&\E \left[\bigl( f(u(s)) - f(u(t_{n+1})), \xi^{n+1} \bigr) \right] \\
&\qquad\leq C \E \left[\| f(u(s)) - f(u(t_{n+1})) \|^2_{L^2} \right]  
+ \E \left[\|\xi^{n+1}\|_{L^2}^2 \right] \notag \\
&\qquad \leq C \tau 
+ \E \left[\|\xi^{n+1}\|_{L^2}^2 \right]. \notag
\end{align}

Next, using properties of the projection, we have
\begin{align}
\label{derrest:9}
& \E \left[\bigl( f(u(t_{n+1}) - f(P_h u(t_{n+1})), \xi^{n+1} \bigr) \right] \\
&\quad =  -\E \bigl[ \bigl( \eta^{n+1} ((u(t_{n+1}))^{q-1}+(u(t_{n+1}))^{q-2}P_h u(t_{n+1})\notag\\
&\qquad+\cdots+P_h u(t_{n+1})^{q-1}-1),\xi^{n+1} \bigr) \bigr] \notag\\
&\quad \leq C \E \Bigl[ \|(u(t_{n+1}))^{q-1}+(u(t_{n+1}))^{q-2}P_h u(t_{n+1}) \notag \\
&\qquad +\cdots+P_h u(t_{n+1})^{q-1}-1\|_{L^{\infty}}^2\times \|\eta^{n+1}\|_{L^2}^2 \Bigr] + \E \left[ \|\xi^{n+1}\|_{L^2}^2 \right] \notag\\
&\quad \leq C\Bigl( \E \left[\left(\|P_h u(t_{n+1})\|_{L^\infty}^{q}
+\|u(t_{n+1})\|_{L^\infty}^{q}+|D|^{\frac{q}{q-1}}\right) \right] \Bigr)^{\frac{q-1}{q}}  \notag \\
&\qquad \times \Bigl( \E \left[\|\eta^{n+1}\|_{L^2}^{2q} \right] \Bigr)^{\frac1q} 
+ \E \left[ \|\xi^{n+1}\|_{L^2}^2 \right] \notag\\
&\quad\leq C\left( \E \left[\|\eta^{n+1}\|_{L^2}^{2q} \right] \right)^{\frac1q} 
+ \E \left[ \|\xi^{n+1}\|_{L^2}^2 \right].\notag
\end{align}

The third term on the right-hand 
side of \eqref{derrest:7} can be bounded by  
\begin{align}
\label{derrest:9_1}
\E \left[\bigl( f(P_h u(t_{n+1})) - f^{n+1}, \xi^{n+1} \bigr) \right] \leq \E \left[\| \xi^{n+1} \|^2_{L^2} \right].
\end{align}

Using Theorem \ref{thm20180802_1}, properties of the interpolation operator, the inverse inequality, and the fact that $u_h^{n+1}$ is a piecewise linear polynomial, the fourth term on the right-hand side of \eqref{derrest:7} can be handled by
\begin{align}\label{eq20180711_17}
&\E \left[\bigl( f^{n+1}-I_hf^{n+1}, \xi^{n+1} \bigr) \right]\\
&\quad\le \E \left[Ch^2\sum_{K\in\mathcal{T}_h}\|q(u_h^{n+1})^{q-1}\nabla u_h^{n+1}\|^2_{L^2(K)} \right]+\E \left[\| \xi^{n+1} \|^2_{L^2} \right]\notag\\
&\quad\le \E \left[Ch^2\sum_{K\in\mathcal{T}_h}\left(\| u_h^{n+1}\|_{L^\infty(K)}^{2(q-1)}\|\nabla u_h^{n+1} \|^2_{L^2(K)}\right) \right]+\E \left[\| \xi^{n+1} \|^2_{L^2} \right]\notag\\
&\quad\le \E \left[Ch^2|\ln h|^2\sum_{K\in\mathcal{T}_h}\left((\|\nabla u_h^{n+1} \|^{2(q-1)}_{L^2(K)}+\|u_h^{n+1} \|^{2(q-1)}_{L^2(K)})\|\nabla u_h^{n+1} \|^2_{L^2(K)}\right) \right]\notag\\
&\qquad \quad+\E \left[\| \xi^{n+1} \|^2_{L^2} \right]\notag\\
&\quad\le \E \left[Ch^2|\ln h|^2\sum_{K\in\mathcal{T}_h}\left(\|\nabla u_h^{n+1} \|^{2q}_{L^2(K)}+\|u_h^{n+1} \|^{2q}_{L^2(K)}\right) \right]+\E \left[\| \xi^{n+1} \|^2_{L^2} \right]\notag\\
&\quad\le \E \left[Ch^2|\ln h|^2(\|u_h^{n+1} \|^{2q}_{L^2}+\|\nabla u_h^{n+1} \|^{2q}_{L^2}) \right]+\E \left[\| \xi^{n+1} \|^2_{L^2} \right]\notag\\
&\quad\le Ch^2|\ln h|^2+\E \left[\| \xi^{n+1} \|^2_{L^2} \right]\notag.
\end{align}

Combining \eqref{derrest:8}--\eqref{eq20180711_17} to obtain
\begin{align} \label{derrest:11}
\E \left[ T_3 \right] &\leq C \tau^{2} 
+Ch^2|\ln h|^2\tau+C \E \left[\|\xi^{n+1}\|_{L^2}^2 \right] \tau\\
&\quad+ C\left( \E \left[\|\eta^{n+1}\|_{L^2}^{2q} \right] \right)^{\frac1q} \tau.\notag
\end{align}

By the martingale property, the It\^{o} isometry, the H\"{o}lder continuity of $u$ and the global Lipschitz condition \eqref{eq20180812_2}, we have
\begin{align}\label{derrest:15}
\E [T_4] &\leq \frac{1}{2} \E \left[\|\xi^{n+1} - \xi^n\|^2_{L^2} \right]+ \frac{1}{2} \E \left[\int_{t_n}^{t_{n+1}} 
\|g(u(s)) -  g(u_h^n)\|^2_{L^2} \, ds \right]  \\
&\leq \frac{1}{2} \E \left[\|\xi^{n+1} - \xi^n\|^2_{L^2} \right]+ C \E \left[\int_{t_n}^{t_{n+1}} 
\|u(s) -  u_h^n\|^2_{L^2} \, ds \right] \notag \\
   &\leq \frac{1}{2} \E \left[\|\xi^{n+1} - \xi^n\|^2_{L^2} \right] + C \E \left[\int_{t_n}^{t_{n+1}} 
\|  u(s) -  u(t_n)\|^2_{L^2} \, ds \right] \notag \\
   &\quad + C \E \left[\| \eta^n +  \xi^n\|^2_{L^2} \right] \tau \notag \\
   &\leq \frac{1}{2} \E \left[\|\xi^{n+1} - \xi^n\|^2_{L^2} \right] 
+ C \tau^2  + C \E \left[\|  \eta^n \|^2_{L^2} \right] \tau\notag \\
   &\quad  + C\E \left[\|  \xi^{n} \|^2_{L^2} \right] \tau.\notag
\end{align}

Taking the expectation on \eqref{derrest:3} and combining estimates \eqref{derrest:4}--\eqref{derrest:15}, summing over 
$n = 0, 1, 2, ..., \ell-1$ with $1 \leq \ell \leq N$, and using the properties of the $L^2$ projection 
and the regularity assumption, we obtain
\begin{align} \label{derrest:17}
\frac14\E \left[\| \xi^{\ell} \|^2_{L^2} \right] 
&+ \frac{1}{4} \E \left[\tau \sum_{n = 1}^{\ell} \| \nabla \xi^n \|^2_{L^2} \right] \\
   & \leq \frac{1}{2} \E \left[ \| \xi^0 \|^2_{L^2} \right] + C\E \left[ \tau \sum_{n=0}^{\ell-1} \| \xi^n \|^2_{L^2} \right]+ C\tau +Ch^2|\ln h|^2.\notag
\end{align}
 
Finally, the assertion of the theorem follows from \eqref{derrest:17}, 
the discrete Gronwall's inequality, the $L^2$-projection properties, the fact that $\xi^0 = 0$ 
and the triangle inequality. The proof is complete. 
\end{proof}

The following strong stability result is a direct corollary of Theorem \ref{thm:derrest}.
\begin{corollary}\label{cor20180803_1}
Suppose the mesh assumption in \eqref{eq20180907} holds and $h^2|\ln h|^2\le C\tau$, then 
\begin{align*}
\E\left[\sup_{0 \leq n \leq N}(\nabla \xi^{n}, \nabla \xi^{n})\right]\le C.
\end{align*}
\end{corollary}

\begin{proof}
For each sample point,
\begin{align}\label{eq20180803_1}
\sup_{0 \leq n \leq N}(\nabla \xi^{n}, \nabla \xi^{n})\le\sum_{n=0}^{N}(\nabla \xi^{n}, \nabla \xi^{n}).
\end{align}
When $h^2|\ln h|^2\le C\tau$, taking the expectation on both sides of \eqref{eq20180803_1}, and using Theorem \ref{thm:derrest}, we obtain
\begin{align*}
\E\left[\sup_{0 \leq n \leq N}(\nabla \xi^{n}, \nabla \xi^{n})\right]\le C+C\frac{h^2|\ln h|^2}{\tau}\le C.
\end{align*}
\end{proof}

\begin{remark}
(a) Notice the elliptic projection cannot be used due to the first term $T_1$ in \eqref{derrest:3}. In reference \cite{majee2018optimal}, it is $C\tau +Ch^2$ since $L^2$ projection is used there.
 
(b) For the diffusion term, We need $g(u)\in C^1$ and $g(u)$ to be Lipschitz continuous, which are the 
same assumptions as in stochastic ODE case \cite{higham2002strong}. The analysis in \cite{majee2018optimal} requires two extra conditions: $g(u)$ 
and $g''(u)$ are bounded. Notice $g(u)=u$, $g(u)=\sqrt{u^2+1}$ or some others satisfy the assumptions 
in this paper, but they do not satisfy the assumptions in \cite{majee2018optimal}.
\end{remark}

\section{Numerical experiments}\label{sec-4}
In this section, we present several two dimensional numerical examples to gauge the performance of the 
proposed stochastic finite element scheme for the stochastic partial differential equations satisfying 
the proposed assumptions for the nonlinear term and the diffusion term. Test 1 is designed to demonstrate 
the error orders with respect to mesh size $h$ for small and big noises; Test 2 is designed to demonstrate 
the stability results and evolution of the stochastic Allen-Cahn equation, which is a 
special case of the SPDE in this paper; Test 3 is designed to demonstrate the stability results of the 
SPDE with a different initial condition; Test 4 is designed to demonstrate the stability results of the 
SPDE with a different nonlinear term; Test 5 is designed to demonstrate the stability results of the SPDE 
with a different diffusion term. The square domain $\Omega = [-1,1]^2$, and 500 sample points are used in 
these tests.

\smallskip
\paragraph{\bf Test 1}
Consider the following smooth initial condition 
\begin{align}\label{test1_init}
u_0(x,y) =\tanh\Big( \frac{x^2 + y^2 - 0.6^2}{\sqrt{2}\epsilon}\Big),
\end{align}
where $\epsilon = 0.2$. Time step size $\tau = 1\times 10^{-6}$ is used in this Test 1.

In this test, the nonlinear term $f(u)=u-u^3$, and the diffusion term $g(u)=\delta\,u$. Table \ref{tab1} shows the following three types of errors $\bigl\{\sup\limits_{0 \leq n \leq N} \E \bigl[ \| e^n \|^2_{L^2(\D)} \bigr]\bigr\}^{\frac12}$, $\bigl\{\E\bigl[\sup\limits_{0 \leq n \leq N} \| e^n \|^2_{L^2(\D)} \bigr]\bigr\}^{\frac12}$, and $\bigl\{\E \bigl[ \sum_{n=1}^N \tau \|\nabla e^n\|^2_{L^2(\D)}\bigr]\bigr\}^{\frac12}$ respectively, and the rates of convergence. 
     The noise intensity $\delta=1$. In the table, we use $L^\infty\E L^2$, $\E L^\infty L^2$ and $\E L^2H^1$ to denote these three types of errors respectively.

\begin{table}[H]
\centering 
\footnotesize
\begin{tabular}{|l||c|c||c|c||c|c|}
\hline  
& $L^\infty\E L^2$ error & order &   
$\E L^\infty L^2$ error & order &   
$\E L^2H^1$ error & order \\ \hline    
$h=0.5\sqrt{2}$ & 0.2909 & --- &   
0.2900 & --- &   
2.2387 & --- \\ \hline    
$h=0.25\sqrt{2}$ & 0.0759 & 1.9384 &   
0.0757 & 1.9377 &   
1.1401 & 0.9735\\ \hline    
$h=0.125\sqrt{2}$ & 0.0201 & 1.9169 &   
0.0201 & 1.9131&   
0.5919 & 0.9457\\ \hline    
$h=0.0625\sqrt{2}$ & 0.0051  & 1.9786 &   
0.0051 & 1.9786&   
0.2996 &  0.9823\\ \hline    
\end{tabular}
\caption{Spatial errors and convergence rates of Test 1: $\epsilon =
  0.2$, $\tau = 1\times 10^{-6}$, $\delta = 1$.} \label{tab1}
\end{table}

Table \ref{tab2} shows the errors $L^\infty\E L^2$, $\E L^\infty L^2$ and $\E L^2H^1$ respectively, and the rates of convergence at final time $T = 2\time10^{-5}$. The noise intensity $\delta=50$.
\begin{table}[H]
\centering 
\footnotesize
\begin{tabular}{|l||c|c||c|c||c|c|}
\hline  
& $L^\infty\E L^2$ error & order &   
$\E L^\infty L^2$ error & order &   
$\E L^2H^1$ error & order \\ \hline    
$h=0.5\sqrt{2}$ & 0.3401 & --- &   
0.2995 & --- &   
2.2708 & --- \\ \hline    
$h=0.25\sqrt{2}$ & 0.0887 & 1.9390 &   
0.0782 & 1.9373 &   
1.1565 & 0.9734\\ \hline    
$h=0.125\sqrt{2}$ & 0.0236 & 1.9101 &   
0.0207 & 1.9175&   
0.6004 & 0.9458\\ \hline    
$h=0.0625\sqrt{2}$ & 0.0060  & 1.9758 &   
0.0053 & 1.9656&   
0.3039 &  0.9823\\ \hline    
\end{tabular}
\caption{Spatial errors and convergence rates of Test 1: $\epsilon =
  0.2$, $\tau = 1\times 10^{-6}$, $\delta = 50$.} \label{tab2}
\end{table}

From these two tables, we observe that the error orders of $L^\infty\E L^2$ and $\E L^\infty L^2$ are 2, and the error order of $\E L^2H^1$ is 1. Besides, the error orders keep the same when the noise intensity increases.

In the following tests, $\E L^2$ and $\E H^1$ are used to denote $\E\|u^{n}_h\|_{L^2}^2$ and $\E \|\nabla u^{n}_h \|^2_{L^2}$ respectively. 

\smallskip
\paragraph{\bf Test 2}\label{test2}
Consider the following initial condition 
\begin{align}\label{test2_init}
u_0(x,y) =\tanh\Big( \frac{\sqrt{x^2 + y^2} - 0.6}{\sqrt{2}\epsilon}\Big).
\end{align}

In this test, the nonlinear term $f(u)=u-u^3$, and the diffusion term $g(u)=\delta\,u$, which corresponds to the stochastic Allen-Cahn equation. More tests related to the Allen-Cahn equation can be found in \cite{feng2014analysis, feng2017finite, li2015numerical, xu2016stability}. Figure \ref{add_evo} shows the evolution of the zero-level sets of the solutions under different intensity of the noise. 
We observe that although the circle may shrink or dilate (depending on the sign of the diffusion term), the average zero-level sets shrink for smaller and bigger noises. Figure \ref{fig2} shows the $\E L^2$ and $\E H^1$ stability results at each time step, which verifies the results in Theorems \ref{thm20180711_1} and \ref{thm20180911}. We also observe that they are both bounded. 
\begin{figure}[!htbp]
\centering 
\captionsetup{justification=centering}
\subfloat[$\delta=0.1$]{
\includegraphics[width=0.40\textwidth]{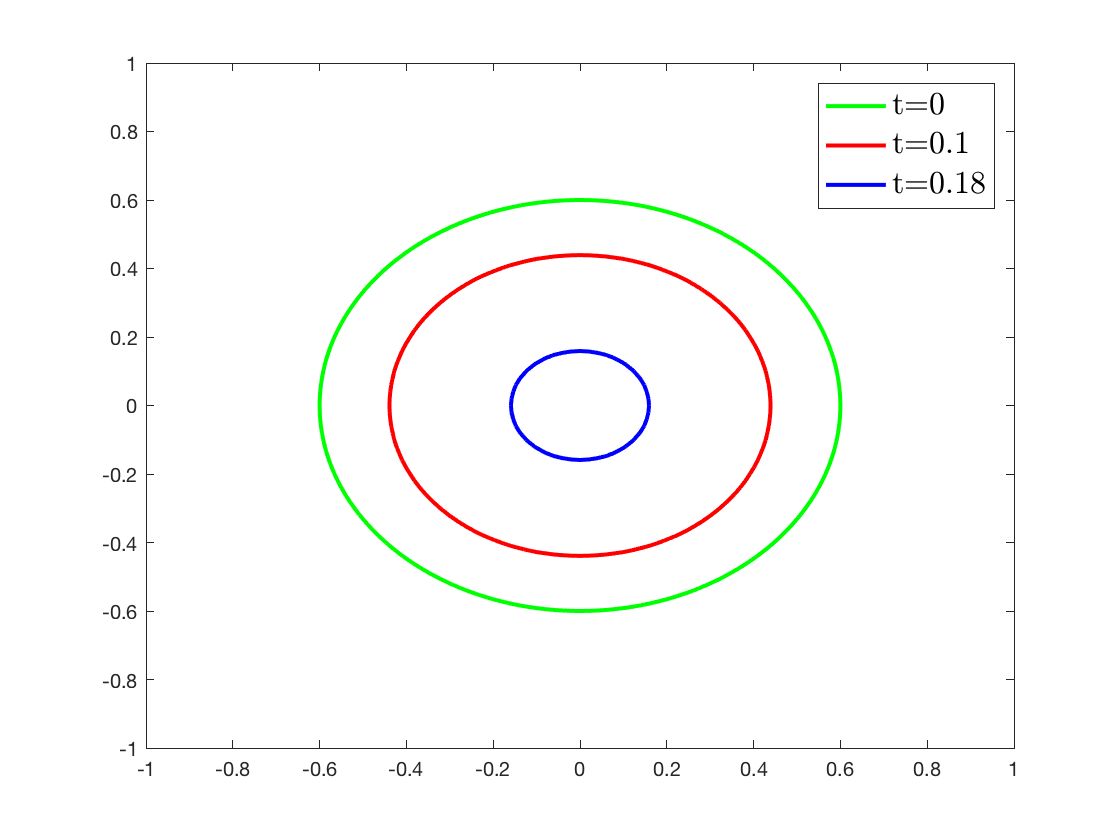}
}%
\subfloat[$\delta=1$]{
\includegraphics[width=0.40\textwidth]{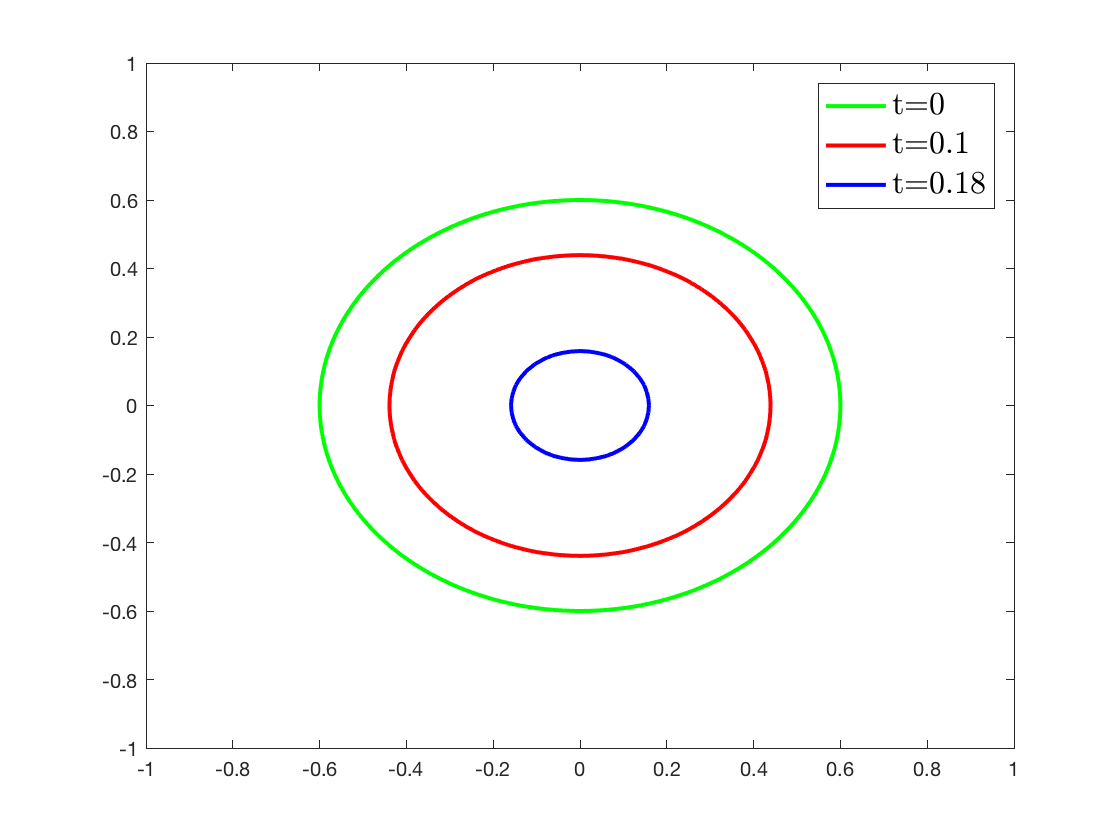}
}
\caption{Zero level sets of the solutions: $\tau=5\times10^{-4},\ h=0.02,\ \epsilon = 0.04$.} 
\label{add_evo}
\end{figure}

\begin{figure}[!htbp]
\centering 
\captionsetup{justification=centering}
\subfloat[$\delta=0.1$]{
\includegraphics[width=0.40\textwidth]{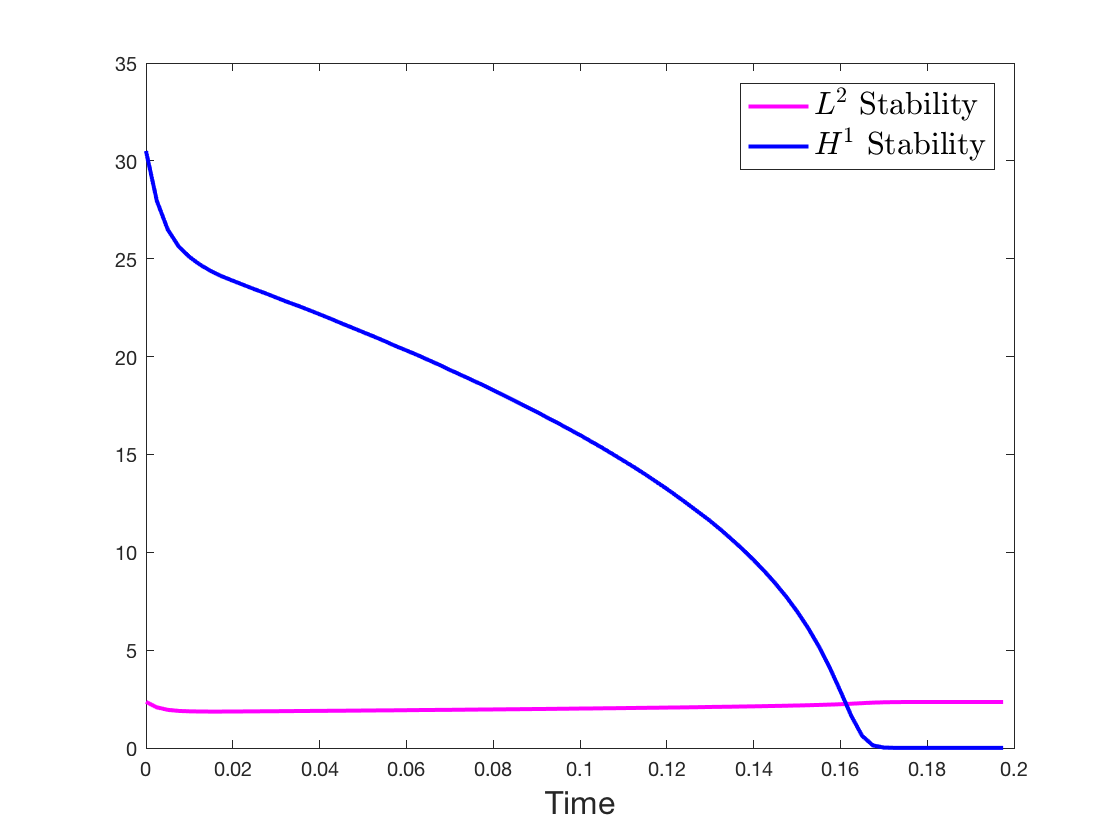}
}%
\subfloat[$\delta=1$]{
\includegraphics[width=0.40\textwidth]{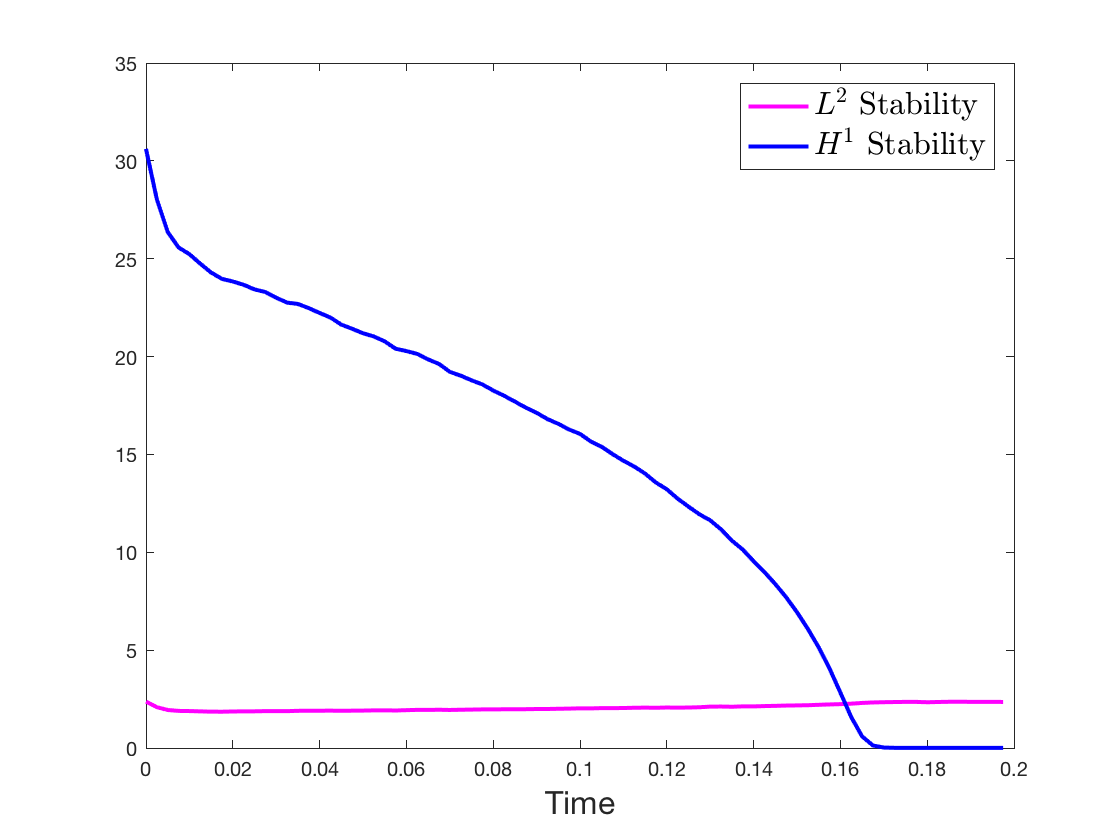}
}
\caption{Stability results: $\tau=2.5\times10^{-3}, \epsilon = 0.1$, and $h=0.04$.}\label{fig2}
\end{figure}

\smallskip
\paragraph{\bf Test 3}
Consider the following initial condition 
\begin{align}\label{test3_init}
u_0(x,y) =\tanh\Big( \frac{1}{\sqrt{2}\epsilon}(\sqrt{x^2/0.04 + y^2/0.36} - 1)(\sqrt{x^2/0.36 + y^2/0.04} - 1)\Big).
\end{align}

In this test, the nonlinear term $f(u)=u-u^3$, and the diffusion term $g(u)=\delta\,u$. Figure \ref{fig3} shows the $\E L^2$ and $\E H^1$ stability results at each time step, which verifies the results in Theorems \ref{thm20180711_1} and \ref{thm20180911}. 

\begin{figure}[!htbp]
\centering 
\captionsetup{justification=centering}
\subfloat[$\delta=0.1$]{
\includegraphics[width=0.40\textwidth]{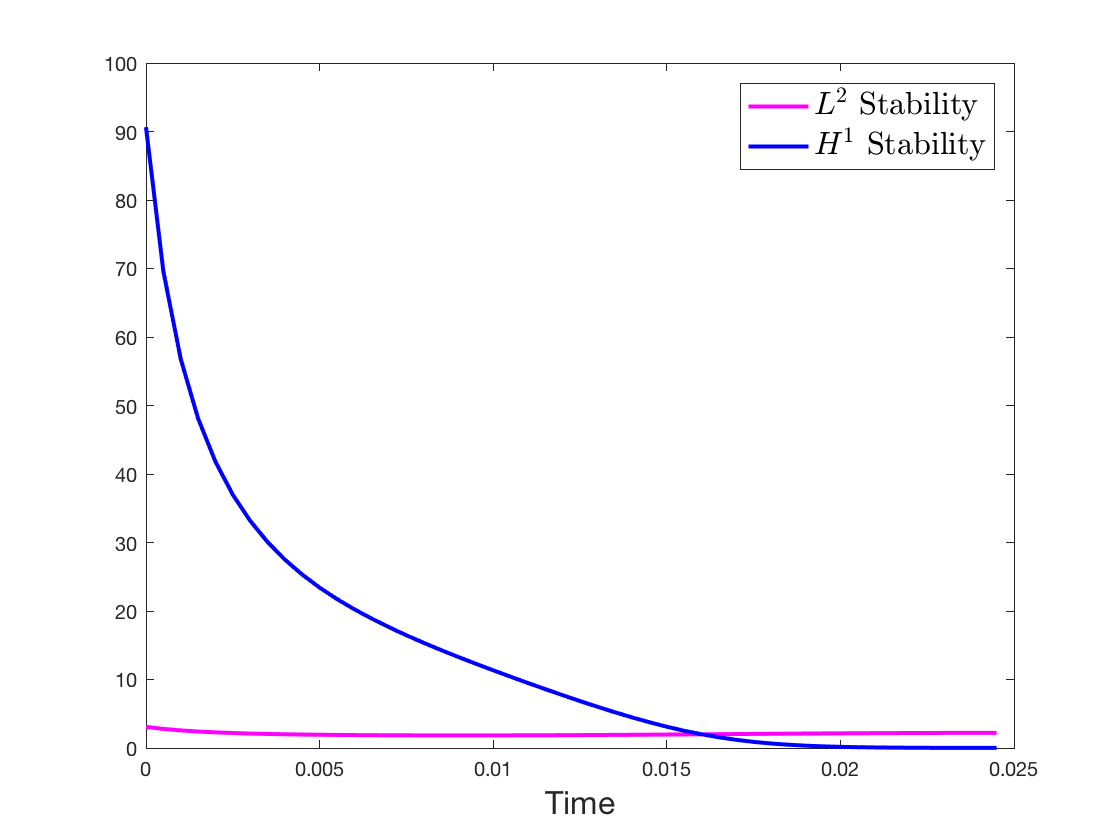}
}%
\subfloat[$\delta=1$]{
\includegraphics[width=0.40\textwidth]{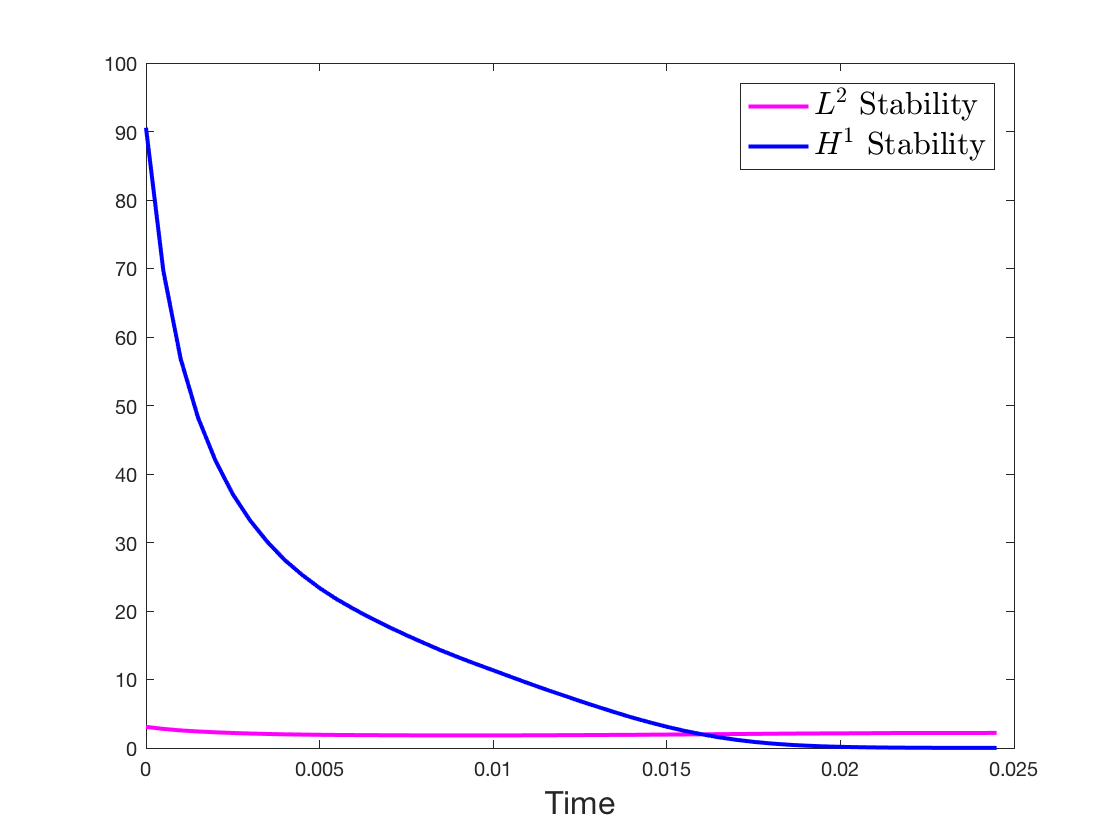}
}
\caption{Stability results: $\tau=5\times10^{-4}, \epsilon = 0.1$, and $h=0.04$.}\label{fig3}
\end{figure}

\smallskip
\paragraph{\bf Test 4}
Consider the initial condition in \eqref{test1_init} with $\epsilon=0.5$.

In this test, the nonlinear term $f(u)=u-u^{11}$, and the diffusion term $g(u)=\delta\,u$. Figure \ref{fig4} shows the $\E L^2$ and $\E H^1$ stability results at each time step, which verifies the results in Theorems \ref{thm20180711_1} and \ref{thm20180911}. 

\begin{figure}[!htbp]
\centering 
\captionsetup{justification=centering}
\subfloat[$\delta=0.1$]{
\includegraphics[width=0.40\textwidth]{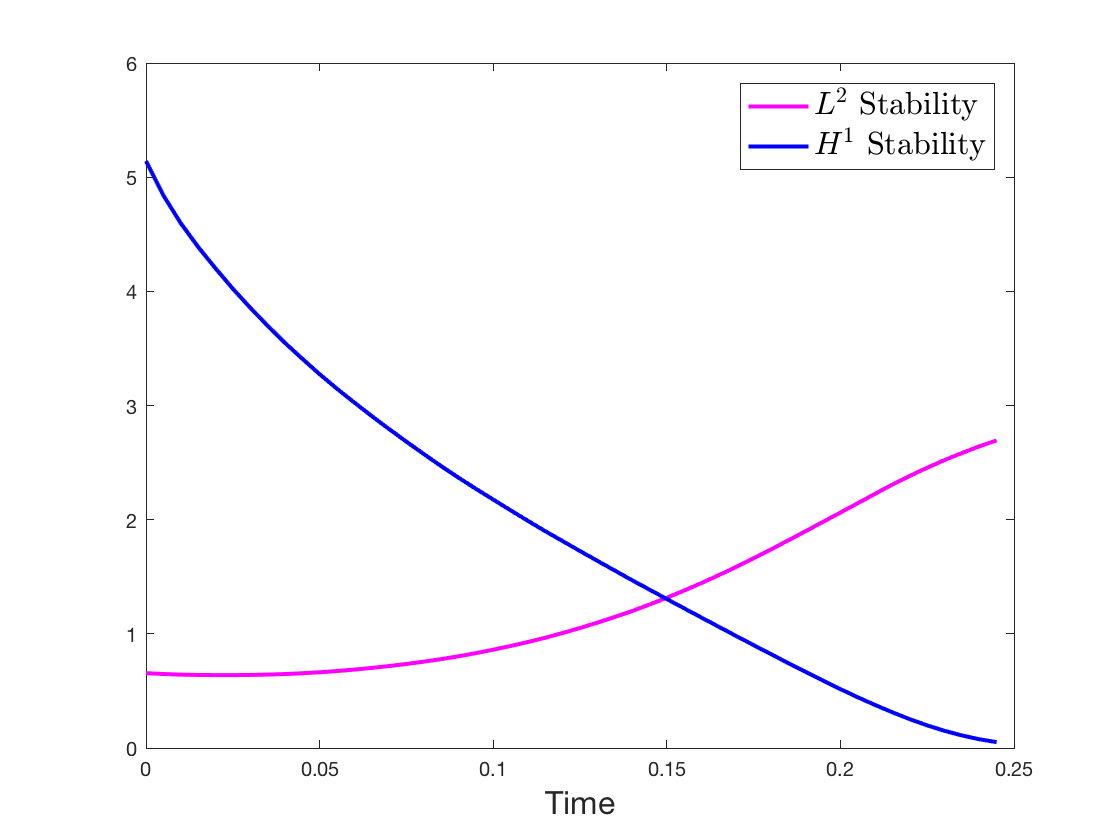}
}%
\subfloat[$\delta=1$]{
\includegraphics[width=0.40\textwidth]{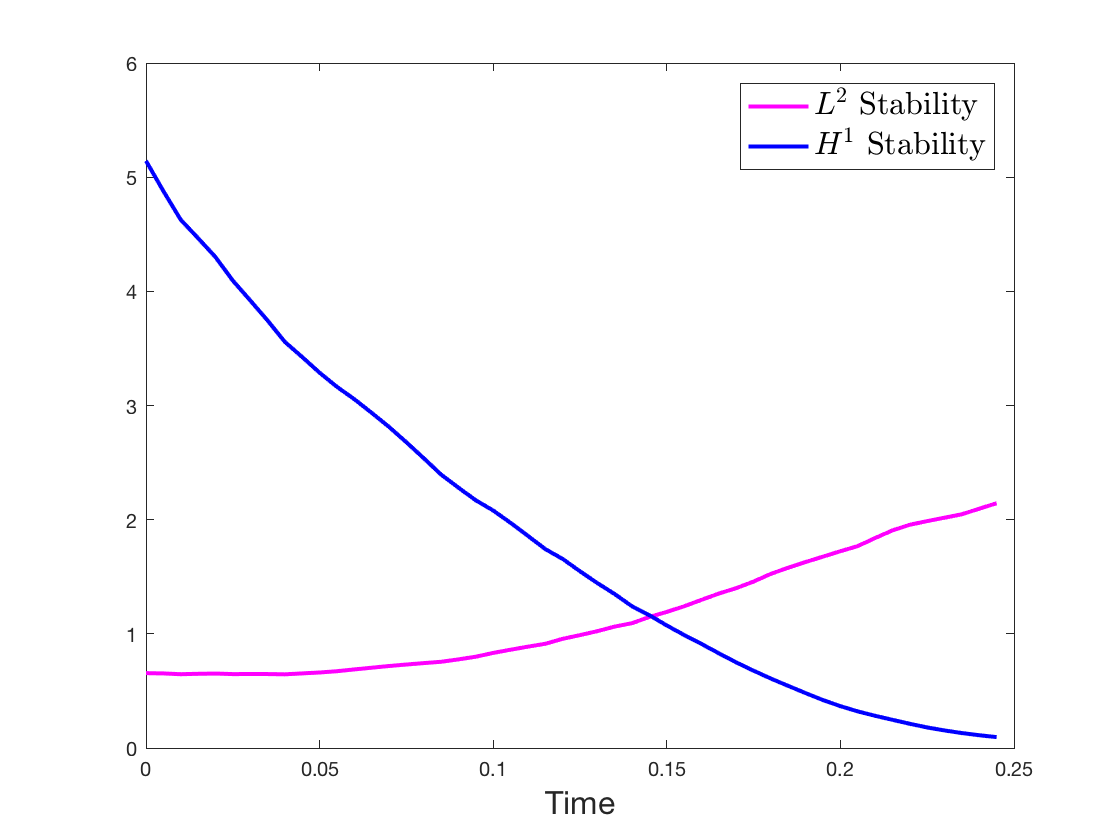}
}
\caption{Stability results: $\tau=5\times10^{-3}, \epsilon = 0.5$, and $h=0.04$.}\label{fig4}
\end{figure}

\smallskip
\paragraph{\bf Test 5}\label{test5}
Consider the initial condition in \eqref{test1_init} with $\epsilon=0.5$.

In this test, the nonlinear term $f(u)=u-u^3$, and the diffusion term $g(u)=\delta\,\sqrt{u^2+1}$. Figure \ref{fig5} shows the $\E L^2$ and $\E H^1$ stability results at each time step, which verifies the results in Theorems \ref{thm20180711_1} and \ref{thm20180911}. 

\begin{figure}[!htbp]
\centering 
\captionsetup{justification=centering}
\subfloat[$\delta=0.1$]{
\includegraphics[width=0.40\textwidth]{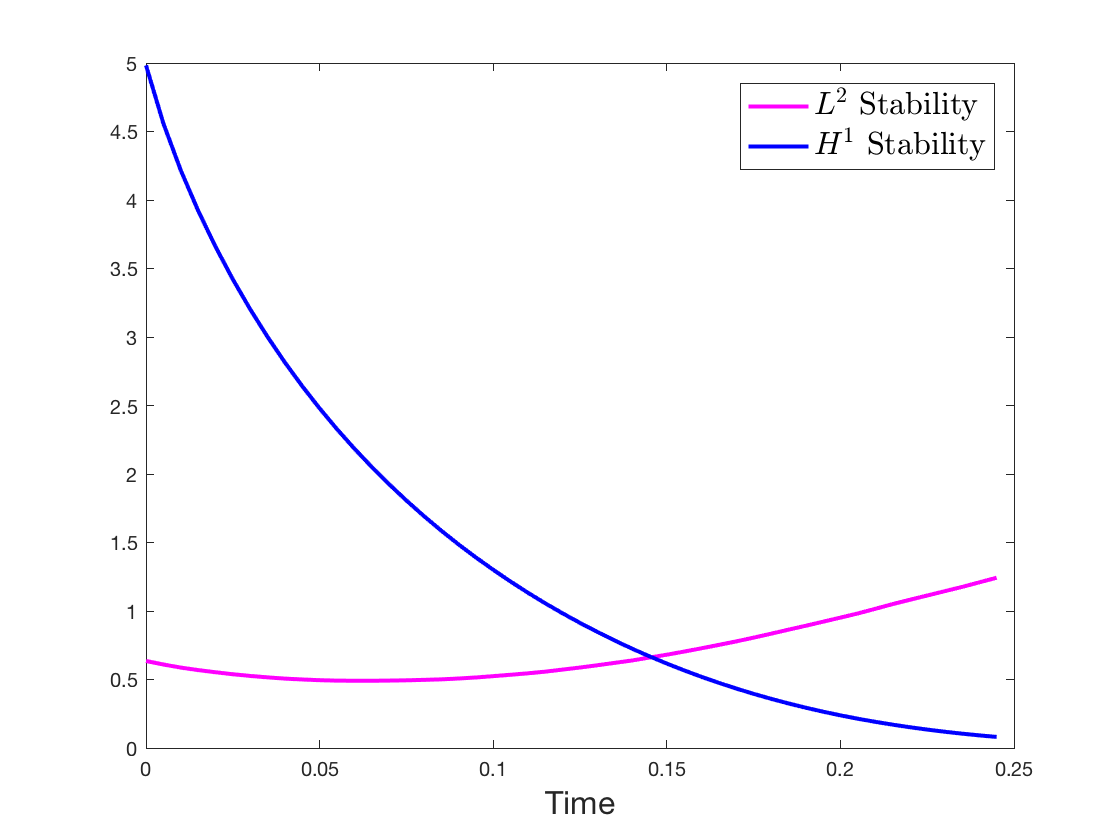}
}%
\subfloat[$\delta=1$]{
\includegraphics[width=0.40\textwidth]{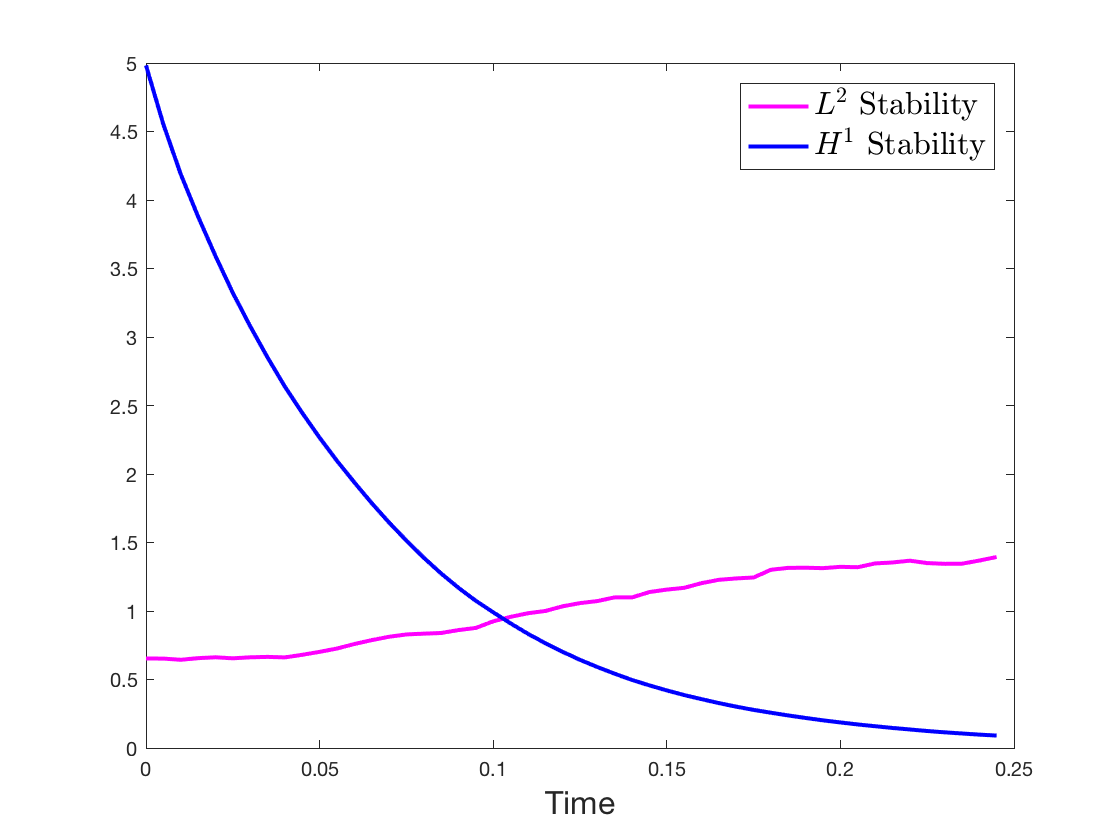}
}
\caption{Stability results: $\tau=5\times10^{-3}, \epsilon = 0.5$, and $h=0.04$.}\label{fig5}
\end{figure}


\bibliographystyle{amsplain}	

\end{document}